\documentclass[11pt,twoside]{amsart}
 \usepackage[latin1]{inputenc}
\usepackage[T1]{fontenc}
\usepackage{amsmath}
\usepackage{amsthm}
\usepackage{amssymb}
\usepackage[all]{xypic}
\usepackage{hyperref}

\newtheorem{thm}{Theorem}[section]

\newtheorem{prop}[thm]{Proposition}

\newtheorem{lem}[thm]{Lemma}
\newtheorem{cor}[thm]{Corollary}
\newtheorem{conj}[thm]{Conjecture}

\numberwithin{equation}{section}

\theoremstyle{definition}
\newtheorem{definition}[thm]{Definition}
\newtheorem{remark}[thm]{Remark}

\newtheorem{exa}[thm]{Example}
\newcommand{\Pic}{{\rm Pic}}

\newcommand{\kc}{{\mathcal C}}
\newcommand{\kh}{{\mathcal H}}

\newcommand{\kl}{{\mathcal L}}

\newcommand{\ko}{{\mathcal O}}
\newcommand{\kx}{{\mathcal X}}

\newcommand{\ZZ}{\mathbb{Z}}
\newcommand{\QQ}{\mathbb{Q}}
\newcommand{\RR}{\mathbb{R}}
\newcommand{\CC}{\mathbb{C}}
\newcommand{\FF}{\mathbb{F}}

\newcommand{\PP}{\mathbb{P}}
\newcommand{\Ch}{{\rm CH}}

\renewcommand{\to}{\xymatrix@1@=15pt{\ar[r]&}}
\renewcommand{\mapsto}{\xymatrix@1@=15pt{\ar@{|->}[r]&}}
\renewcommand{\twoheadrightarrow}{\xymatrix@1@=15pt{\ar@{->>}[r]&}}
\renewcommand{\hookrightarrow}{\xymatrix@1@=15pt{\ar@{^(->}[r]&}}
\newcommand{\congpf}{\xymatrix@1@=15pt{\ar[r]^-\sim&}}
\renewcommand{\cong}{\simeq}

\begin{document}

\title{Curves and cycles on K3 surfaces}
\author[D.\ Huybrechts]{D.\ Huybrechts\\ with an appendix by C.\ Voisin}

%\address{Mathematisches Institut,
%Universit{\"a}t Bonn, Endenicher Allee 60, 53115 Bonn, Germany}
%\email{huybrech@math.uni-bonn.de}

\begin{abstract} \noindent
  The notion of constant cycle curves on K3 surfaces is introduced. These are curves that do not
 contribute to the Chow group of the ambient K3 surface. Rational curves are the most prominent examples.
 We show that constant cycle curves behave in some respects like rational curves. E.g.\ using Hodge theory
 one finds that in each linear system there are at most finitely many such curves of bounded order.
 Over finite fields, any curve is expected to be a constant cycle curve, whereas over $\bar\QQ$
 this does not hold. The relation to the Bloch--Beilinson conjectures for K3 surfaces over global fields
 is discussed.
 \vspace{-2mm}
\end{abstract}

\maketitle
{\let\thefootnote\relax\footnotetext{This work was supported by the SFB/TR 45 `Periods,
Moduli Spaces and Arithmetic of Algebraic Varieties' of the DFG
(German Research Foundation).}
%\marginpar{}
}

\tableofcontents

\newpage
%\newpage
\section{Introduction}

\subsection{} Due to results of Mumford \cite{Mumford} and Bloch \cite{Bloch}, the Chow group of zero-cycles
$\Ch_0(X)=\Ch^2(X)$ on a complex K3 surface $X$ is
known to be huge (infinite dimensional in some well-defined sense). In particular, there is no curve $C\subset X$ such that
the natural push-forward map
\begin{equation}\label{eqn:push0}
\Pic(C)\cong\Ch_0(C)\to\Ch^2(X)
\end{equation}
 is surjective. This paper studies  curves
$C\subset X$ for which the image of $\Pic(C)\to\Ch^2(X)$ is as small as possible, i.e.\
for which (\ref{eqn:push0}) induces a trivial map $\Pic^0(C)\to \Ch^2(X)$.

For a K3 surface over an algebraically closed field $k$, we define an integral curve $C\subset X$ to be a \emph{constant cycle curve} if 
the class of 
its generic point $\eta_C\in C$ viewed as a closed point in $X_{k(\eta_C)}=X\times_kk(\eta_C)$ satisfies 
\begin{equation}\label{eqn:master}
n\cdot[\eta_C]=n\cdot(c_X)_{k(\eta_C)}
\end{equation}
in $\Ch^2(X\times_k{k(\eta_C)})$ for some positive integer $n$. Here, $c_X\in\Ch^2(X)$ is the distinguished class of degree one introduced
by Beauville and Voisin in \cite{BV}. The minimal such $n$ is called the \emph{order} of the constant cycle curve.
If $C\subset X$ is a constant cycle curve, then $\Pic^0(C)\to{\rm CH}^2(X)$ will be shown to be indeed the zero map.
See Sections \ref{sec:pccc} and \ref{sec:defkappa} for the definitions and Proposition \ref{prop:Voisin} for the relation between the two notions.

Finding sufficient criteria that decide whether a given curve is a constant cycle curve seems as hard as finding criteria that
would ensure the opposite. The only positive criterion at the time being is that any rational curve is a constant cycle curve 
(of order one) and the only effective method to 
exclude a given curve from being a constant cycle curve uses Hodge theory (cf.\ Corollary \ref{cor:kappaAJ}).

\subsection{} Some of the results proved in this article are:
\smallskip

$\bullet$ \emph{There are at most finitely many constant cycle curves of bounded order in each linear system $|L|$ on a K3 surface
$X$ in characteristic zero (cf.\ Proposition \ref{prop:finite}).}
\smallskip

The enumerative problem that suggests itself at this point remains largely open, but see Section \ref{sec:counting}.

\smallskip

To get a better idea of the notion of a constant cycle curve, we give many concrete examples with an emphasis
on  curves of low order and high genus, see Sections \ref{sec:exas}--\ref{sec:bitangentcorr}. 
Apart from curves of torsion points in families of elliptic curves, the following result
turns out to be useful.

\smallskip

$\bullet$ \emph{Every fixed curve of a non-symplectic automorphism is a constant cycle curve (see Proposition \ref{prop:fixedcurve}).}

\smallskip

This immediately leads to:

\smallskip

$\bullet$ \emph{There are constant cycle curves of order one, that are not rational (see Corollary \ref{cor:cccorderone}).}

\smallskip

We also manage to construct a constant cycle curve in the generic quartic hypersurface $X\subset\PP^3$
that is of order at most four and genus $201$, see Proposition \ref{cor:cccbit}.

The notion of constant cycle curves makes sense for arbitrary surfaces. In Section \ref{sec:BlochConj} we briefly study surfaces with $p_g=0$ satisfying Bloch's conjecture. It can easily be shown that
curves on such surfaces are all constant cycle curves of bounded order and 

\smallskip

$\bullet$ \emph{On an Enriques surface all curves are constant cycle curves of order at most four (see Proposition \ref{prop:ccconrat}).}

\smallskip

In Section \ref{sec:finitefields} we discuss K3 surfaces over finite fields  and prove:

\smallskip

$\bullet$ \emph{For a Kummer surface $X$ over $\bar\FF_p$ every curve in $X$ is a constant cycle curve (cf.\ Proposition
\ref{cor:allccconKummer}).}

\smallskip

This is expected to hold for all K3 surfaces over $\bar\FF_p$ and is related to the conjectured
finite-dimensionality in the sense of Kimura--O'Sullivan and the Bloch--Beilinson conjecture for K3 surfaces over
global fields in positive characteristic (see Proposition \ref{prop:BBCff}). For arbitrary K3 surfaces
over $\bar\FF_p$ one can at least prove the following:

\smallskip

$\bullet$ \emph{Let $X$ be a K3 surface over $\bar\FF_p$. Then every closed point $x\in X$
is contained in a constant cycle curve (see Proposition \ref{prop:cccallover}).}

\smallskip

This is expected to hold as well for $X$ over $\bar\QQ$ and would imply the Bloch--Beilinson conjecture
for $X$. For arbitrary K3 surfaces the existence is expected for points rationally equivalent to points on rational curves.
The result should also be compared with \cite[Thm.\ 4.2]{BT2}, where it is shown that every point in a Kummer surface over $\bar\FF_p$
associated to the Jacobian of a curve of genus two  is contained in a rational curve (and hence in a constant cycle curve of order one).

%%%%%%%%%%%%
\subsection{} {\it Acknowledgements:} 
I am grateful to Claire Voisin for  many comments, in particular
on the torsion problem in the Bloch--Srinivas argument, and for the example in Section \ref{sec:infinell}.
Thanks to Burt Totaro for a question that triggered the results in Section
\ref{sec:finitefields} and for detailed comments on the first version,
to Rahul Pandharipande for bringing \cite{GG} to my attention,
and to Fran\c{c}ois Charles and
Davesh Maulik for help with arguments related to Section \ref{sec:bitangentcorr}. 
Jimmy Dillies and Alessandra Sarti  patiently  answered my  emails concerning Section \ref{sec:auto} and
Kieran O'Grady  commented on the content of Section \ref{sec:bitangentcorr}.
Suggestions of the referee have helped to improve the exposition.
The intellectual debt to the foundational work of Mark Green, Phillip Griffiths, and Claire Voisin is gratefully acknowledged.

%%%%%%%%%%%%%%%%%%%%%%%%%%%%%%%
\section{Motivation}

We shall try to motivate the study of constant cycle curves from two perspectives: rational curves and Chow groups. 
For simplicity we shall restrict to K3 surfaces over algebraically closed fields $k$ with ${\rm char}(k)=0$ and in fact $k=\bar\QQ$ or $\CC$.
For technical details, in particular concerning the case ${\rm char}(k)>0$, see the later sections.

\subsection{}\label{sec:BMMM} Let  $(X,H)$ be a polarized complex K3 surface. The following  folklore conjecture
has been studied intensively over the last couple of years.

\begin{conj}\label{conj:dense} 
The union $\bigcup C\subset X$ of all rational curves $C\subset X$ is dense (in the Zariski or, stronger, in the classical topology).
\end{conj}

The stronger and less studied version would only allow integral rational curves linearly equi\-valent
to some multiple of the given polarization $H$.

The motivation for this conjecture stems from the classical result that for all $m>0$ there exists a rational curve $C$ in 
the linear system $|mH|$ (Bogomolov, Mumford, Mori--Mukai \cite{MM}). Here, a curve is rational if the reduction
of each of its components has a normalization $\cong\PP^1$.
For fixed $m>0$ and \emph{generic} $(X,H)$, i.e.\ for polarized complex K3 surfaces
in a Zariski dense open subset of the moduli space of polarized K3 surfaces, $C$ can be chosen integral and
even nodal \cite{Chen}. Note that at the same time there are at most finitely many rational curves in any fixed linear system, e.g.\ in $|mH|$, as K3 surfaces in characteristic zero are not unirational.

More recently, the conjecture (for the Zariski topology) has been verified for K3 surfaces with $\rho(X)\equiv1\,(2)$ by
Li and Liedtke \cite{LL} following an approach by Bogomolov--Hassett--Tschinkel \cite{BHT}. The same ideas 
 also apply to K3 surfaces that are not defined over $\bar\QQ$ (see \cite{HK3} for details). Note that both conditions,
`$\rho(X)\equiv\,1(2)$' and `not defined over $\bar\QQ$', are \emph{general} but not generic, i.e.\ they hold for K3 surfaces
in the complement of a countable union of proper algebraic subsets of the moduli space of polarized K3 surfaces.
Using work of Bogomolov and Tschinkel \cite{BT}, Chen and Lewis \cite{CL} settled the conjecture in the classical topology for
general $(X,H)$.
Not much is known about the stronger form of the conjecture, except for $\rho(X)=1$ (when the rational curves have no other choice than being linearly equivalent to multiples of $H$).

\subsection{}\label{sec:IntroBV} Assume $X$ is a K3 surface over $\CC$ (or $\bar\QQ$).
In \cite{BV} Beauville and Voisin described a distinguished class $c_X\in\Ch^2(X)$ of degree one, which in particular
has the properties that 
\begin{equation}\label{eqn:BV}
{\rm c}_2(X)=24\cdot c_X\text{ and }{\rm c}_1(L)^2\in\ZZ \cdot c_X
\end{equation}
for all line bundles $L$ on $X$.
The set of closed points realizing this class was subsequently studied by McLean in \cite{ML}, 
where it is shown that the set
\begin{equation}\label{eqn:McLean}
X_{c_X}:=\{x\in X~|~[x]=c_X\in\Ch^2(X)\}\subset X
\end{equation}
is dense in the classical topology.
Similarly, one can consider the set $X_\alpha$ of points realizing any given class $\alpha\in\Ch^2(X)$ and again, at least for generic
K3 surface, this set is dense if not empty. However, as shall become clear, the set $X_{c_X}$ is rather special. For abstract reasons, it is a countable union of Zariski closed subsets, but one expects it to be a countable union of curves, i.e.\ isolated points should not occur. This 
would be another distinction between $c_X$ and any other class $\alpha\in\Ch^2(X)$
(cf.\  \cite{VoisinHOG}). See also Section \ref{McLeanunderbasechange} for more on the sets $X_{[x]}$.

\smallskip

The distinguished class $c_X$ can also be considered
from a more arithmetic point of view, as expressed by the following special case of the much more
general set of conjectures due to Bloch and Beilinson.
But note that  even for K3 surfaces, it has not been verified in a single example.

\begin{conj}\label{conj:BBC}(Bloch--Beilinson) Suppose $X$ is a K3 surface over $\bar \QQ$ and $x\in X(\bar\QQ)$ is a $\bar\QQ$-rational point.
Then $[x]=c_X$.
\end{conj}

Note that McLean's proof in fact shows that for every K3 surface $X$ defined over some subfield $k\subset \CC$ the set
of $\bar k$-rational points realizing $c_X$ is dense in the classical topology, i.e.\ $X_{c_X}(\bar k)\subset X(\CC)$ is dense.
In particular, it is known that for $X$ over $\bar \QQ$ there are many points $x\in X(\bar\QQ)$ realizing $c_X$.
\smallskip

The property of $c_X$ proved in \cite{BV} that is the most relevant for our purpose, is the following: 

\begin{equation}\label{eqn:BVmostrelevant}
\text{ If }x\in C\subset X\text{  with } C \text{ rational, then } [x]=c_X.
\end{equation}
%\smallskip
This links rational curves to the study of $\Ch^2(X)$ and the distinguished class $c_X\in\Ch^2(X)$. In particular,
McLean's density result could be seen as (weak) evidence for the density for rational
curves as in Conjecture \ref{conj:dense}. Also, as pointed out by Bogomolov many
years ago, it might a priori be possible that any $\bar\QQ$-rational point lies on a rational curve which in turn would prove
Conjecture \ref{conj:BBC}.

The fact that points on rational curves all define the same class in $\Ch^2(X)$, which eventually relies on the existence
of ample rational curves \`a la Bogomolov--Mumford \cite{MM}, also leads to the concept of {constant cycle curves} studied 
in this paper.
 
\subsection{}
Let $C\subset X$ be a curve in a complex K3 surface. Then $C$ is called a \emph{(pointwise)
constant cycle curve} 
if $[x]\in\Ch^2(X)$ is constant for points $x\in C$ or, equivalently, if the push-forward $\Pic^0(C)\to\Ch^2(X)$ is the zero map. 
Voisin shows in \cite{VoisinHOG}, by again using the existence of ample rational curves, that the class realized by a constant cycle curve is always the same, namely $c_X$.
The most important examples of constant cycle curves are provided by rational curves. But not every 
(pointwise) constant cycle curve is rational, see Section \ref{sec:exas}-\ref{sec:bitangentcorr} for examples.
This triggers the natural question how much weaker the notion of constant cycle curves really is.

As rational curves, constant cycle curves do not come in families (at least not in characteristic zero). Indeed, any family of constant cycle curves
would dominate $X$ and so points in an open dense subset would all
realize the same class in $\Ch^2(X)$ contradicting $\Ch^2(X)\ne\ZZ$ (cf.\ proof of Lemma \ref{lem:ccc}, ii)).
 Hence, for abstract reasons, the set of constant cycle curves in a fixed linear system, e.g.\ in $|mH|$, consists of at most countably many points. A finiteness result as for rational curves can be proved after restricting to constant cycle curves of bounded order
(see Proposition \ref{prop:finite}).
This result is based on normal functions and the recent results of Brosnan--Pearlstein \cite{BP}
and Saito \cite{Saito} showing that the zero-set of admissible normal
functions is algebraic. 

Alternatively to the definition of the order of a constant cycle curve using (\ref{eqn:master}) one
could define it directly as the order  of  a certain class
 $\kappa_C\in\Ch^2(X\times{k(\eta_C)})$ naturally associated
 to any integral curve $C\subset X$ with    its generic point $\eta_C\in C$ (see Section \ref{sec:defkappa}).
 Note that $\Ch^2(X)$ is torsion free for $X$ over a separably closed field and so the subtle information needed for the finiteness is contained
in the kernel of $\Ch^2(X\times{k(\eta_C)})\to\Ch^2(X\times{\overline{k(\eta_C)}})$. 

It is not difficult to show that rational curves are in fact constant cycle curves of order one (see Lemma \ref{sec:ratcurves}). Also,
non-rational constant cycle curves can be constructed in many ways, but they usually tend to be of higher order,
i.e.\ $\kappa_C\ne0$ in $\Ch^2(X\times{k(\eta_C)})$. So it is natural to wonder whether
constant cycle curves of order one are all rational, but this turns out to be wrong and  an explicit counterexample will be described
(see Corollary \ref{cor:cccorderone}).

As constant cycle curves of bounded order resemble  rational curves in many ways, we
state Conjecture \ref{conj:dense} for this more flexible class of curves. Note that if in the following the order is not bounded, the result is not difficult to prove, see \cite{VoisinHOG} or Lemma \ref{lem:ccctorsion}.
\begin{conj}\label{conj:ccc}
For any K3 surface $X$ there exists an $n>0$ such that the union $\bigcup C\subset X$ of all constant cycle curves $C\subset X$ of order $\leq n$ is dense.\footnote{In the appendix Claire Voisin provides a proof of the conjecture for generic complex K3 surfaces. The main idea is to produce constant cycle curves as non-torsion multi-sections of dominating
families of elliptic curves, similar to \cite{BT}.}
\end{conj}
Once the finiteness of constant cycle curves of bounded order has been established, it would be interesting
to actually count them. Counting rational curves on K3 surfaces is a fascinating subject which
recently culminated in the proof of the Yau--Zaslow conjecture in complete gene\-ra\-lity in \cite{KMPS}. Unfortunately, it seems much harder  to count constant cycle curves, see Section \ref{sec:counting}. 

There is little evidence for an affirmative answer to Bogomolov's question whether maybe any $x\in X(\bar\QQ)$
is contained in a rational curve. Again, one could replace rational curves by constant cycle curves and
an affirmative answer to this weaker form would still imply the arithmetic Conjecture \ref{conj:BBC}. 
The following could be seen as a geometric version.

\begin{conj}\label{conj:BBCgeom}
Let $X$ be a complex K3 surface. Then any point $x\in X$ with $[x]=c_X$ is contained in a constant
cycle curve.
\end{conj}

As alluded to before, it is much easier to construct constant cycle curves than rational curves.
E.g.\ for K3 surfaces over finite fields every point is in fact contained in a constant cycle curve
(see Proposition \ref{prop:cccallover}).
However, a general technique that would allow to settle this problem is not yet available.

%%%%%%%%%%%%%%%%%%%%%%%%%%%%%

\section{Constant cycle curves}\label{sec:deficcc}

\subsection{}\label{sec:pccc}
Let $X$ be a projective K3 surface over a field $k$. For two reasons, one often has
to assume that  $k$  is algebraically closed. Firstly, $\Ch^2(X)$ might have torsion otherwise
and, secondly, the good behavior of the Beauville--Voisin class $c_X$ 
depends on the existence of rational curves for which $k$ algebraically closed is needed.
In fact, for $k$ not algebraically closed, $c_X\in\Ch^2(X)$ has the desired properties (e.g.\
being realized by points on rational curves) only up to torsion. 
We will state explicitly when $k=\bar k$ is assumed.

\begin{definition}
A curve $C\subset X$ is a \emph{pointwise constant cycle curve} if all closed points $x\in C$ 
define the same class $[x]\in\Ch^2(X)$.
\end{definition}

For $k=\bar k$ the condition is equivalent to require
$[x]=c_X$ (the Beauville--Voisin class, see Section \ref{sec:IntroBV})
for all closed points $x\in C$, see \cite[Lem.\ 2.2]{VoisinHOG}.\footnote{We tacitly assume that the
standard facts on $c_X$ hold true for K3 surfaces over algebraically closed fields of positive characteristic, but 
all we really need is that points on rational curves all realize the same class. This follows
from the existence of ample rational curves which can be shown by reduction modulo $p$.}
Another way of expressing the condition (still assuming $k=\bar k$) is as follows. A
curve $C\subset X$ is a pointwise constant cycle curve if
and only if the natural map 
\begin{equation}\label{eqn:push}f_{C*}:\Pic(\widetilde C)\to \Ch^2(X)
\end{equation}
takes image in $\ZZ \cdot c_X$ or, still equivalent, that 
\begin{equation}\label{eqn:push2}f_{C*}:
\Pic^0(\widetilde C)\to\Ch^2(X)
\end{equation}
is zero. Here,  $f_C:\widetilde C\to X$ is the composition of the normalization
$\widetilde C\to C$  with the inclusion $C\subset X$.

The notion of pointwise constant cycle curves is really interesting only for uncountable fields $k$.
E.g.\ it is not preserved under base change when the base field $k$ is too small
(cf.\ Lemma \ref{lem:ccc} and Proposition \ref{prop:Voisin}).
Also, according to Conjecture \ref{conj:BBC}, every curve in a K3 surface over $\bar\QQ$ should be a pointwise constant cycle curve which makes it a notion of little interest in this case. However, the same is true
for K3 surfaces over $k=\bar\FF_p$ and in this case it is a shadow of the Bloch--Beilinson conjecture
for function fields, see Proposition \ref{prop:BBCff}.

%%%%%%%%%%%%%%%%%%
\subsection{}\label{sec:defkappa}
In order to introduce the finer version of this notion, we  define
the class $\kappa_{C}$ naturally associated to any integral curve  $C\subset X$. To this end,
denote (abusively) by $\Delta_C\subset X\times C$ the graph
of the inclusion and consider the cycle $\Delta_C-\{x_0\}\times C$, where $x_0\in X$ is an arbitrary
point with $[x_0]=c_X$.  Again, for the existence of such a point and for $c_X$ being well-defined, $k$ has to be algebraically
closed. Here, $c_X$ is the Beauville--Voisin class (cf.\ Section \ref{sec:IntroBV}). Then let 
\begin{equation}\label{eqn:kappa}
\kappa_{C}\in\Ch^2(X\times{k(\eta_C)})
\end{equation} be the class of the restriction
$\Delta_C-\{x_0\}\times C$ to the generic fibre $X_{k(\eta_C)}=X\times_kk(\eta_C)$ of the second projection
$X\times C\to C$. In other words, the generic point $\eta_C\in C$  is viewed as a closed point of
the K3 surface $X_{k(\eta_C)}$ over the function field $k(\eta_C)$ of $C$
and then corrected by the `constant' point $\{x_0\}\times \eta_C$.
This natural class has been considered before in the literature, see e.g.\ \cite[Ex.\ 4.2]{Kerr2}.

\begin{definition} Assume $k$ algebraically closed.
An integral curve $C\subset X$ is a \emph{constant cycle curve} if $\kappa_C\in\Ch^2(X\times{k(\eta_C)})$ is a torsion class. 
\end{definition}

We call an arbitrary curve $C\subset X$  a constant cycle curve if every integral component of $C$
has this property. Note that this definition makes perfect sense for all surfaces with a distinguished class in $\Ch^2(X)$, e.g.\
for those with $\Ch^2(X)\cong\ZZ$. In fact constant cycle curves can be defined for arbitrary surfaces by means
of Lemma \ref{lem:BVclass2}, but, with the exception of Section \ref{sec:BlochConj}, we will restrict to K3 surfaces.

\begin{remark} The class $\kappa_C$ is in fact the direct image under the push-forward
$$\Ch^1(C\times k(\eta_C))\to\Ch^2(X\times k(\eta_C))$$ of the class $[\eta_C]-[x_0]$, where
$x_0\in C$ is any point with $[x_0]=c_X$ in $\Ch^2(X)$ (e.g.\ a point of intersection with  a
rational curve). However, the class $[\eta_C]-[x_0]\in\Ch^1(C\times k(\eta_C))$ itself is never torsion except for $C$ rational, because
no non-trivial multiple of $[\eta_C]$ is ever contained in the image of the base change map $\Ch^1(C)\to\Ch^1(C\times k(\eta_C))$.

In fact, whether a non-rational curve $C$ is a constant cycle curve depends on the 
particular embedding $C\,\hookrightarrow X$. We will see examples of (smooth)
curves that can be embedded as constant cycle curves of varying
order and even as non-constant cycle curves in the same K3 surface $X$.
\end{remark}

For any field extension $K/k$ the pull-back yields a map
\begin{equation}\label{eqn:pb2}
\Ch^2(X)\to\Ch^2(X\times_k K).
\end{equation}
The image of the Beauville--Voisin class $c_X\in\Ch^2(X)$
shall be denoted $(c_X)_K$. It can also be seen, at least for $K$ algebraically closed,
as the Beauville--Voisin class of $X_K=X\times_kK$, i.e.\ $(c_X)_K=c_{X_K}$.
Compare the following result to Lemma \ref{lem:BVclass2}.

\begin{lem}\label{lem:BVclass} Let $X$ be a K3 surface over an algebraically closed field $k$.
For an integral curve $C\subset X$ the following conditions are equivalent:

i) The curve $C$ is a constant cycle curve.

ii) There exists a positive integer $n$ such that 
\begin{equation}\label{eqn:minimal}
n\cdot[\eta_C]=n\cdot (c_X)_{k(\eta_C)}
\end{equation}
in $\Ch^2(X\times_k{k(\eta_C)})$, where the generic point
$\eta_C\in C$ is viewed as a closed point in $X\times_k{k(\eta_C)}$. 

iii) If  $\eta_C\in C$ is viewed as a point in
the geometric generic fibre $X\times_k{\overline {k(\eta_C)}}$, then
$$[\eta_C]=(c_X)_{\overline{k(\eta_C)}}$$
in $\Ch^2(X\times_k{\overline {k(\eta_C)}})$. 
\end{lem}

\begin{proof}
This is an immediate consequence of the fact that the pull-back
(\ref{eqn:pb2}) has torsion kernel and Roitman's theorem \cite{Roitman}, and its generalizations due to Bloch
\cite{Bloch} and Milne \cite{MilneRoit},
showing that  the group $\Ch^2(X\times_k{\overline {k(\eta_C)}})$  is torsion free.
\end{proof}

The Chow group of the generic fibre is best 
 viewed as
\begin{equation}\label{eqn:dirlimi}
\Ch^2(X\times{k(\eta_C)})=\lim\limits_{\rightarrow}\Ch^2(X\times U),
\end{equation}
where the direct limit is over all non-empty Zariski open subsets $U\subset C$,
see \cite[Lem.\ 1.I.20]{Bloch}. Hence, $\kappa_C$ is a torsion class
if and only if the class $\kappa_{C,U}\in\Ch^2(X\times U)$
of the restriction of $[\Delta_C-\{x_0\}\times C]$ to $X\times U$ for some non-empty open subset $U\subset C$ is
torsion. 

\smallskip

 In order to obtain finiteness results, one needs to bound the order of the torsion class
$\kappa_C$ of a constant cycle curve.
\begin{definition}
The \emph{order} of an integral constant cycle curve $C\subset X$ is the order of the torsion class $\kappa_C\in
\Ch^2(X\times{k(\eta_C)})$. The order of an arbitrary constant cycle curve $C$ is the maximal order of its integral components.
\end{definition}

Note that by shrinking $U\subset C$ one can always assume that the order of $\kappa_C$ and 
$\kappa_{C,U}$ coincide. By definition the order is the minimal positive $n$ satisfying (\ref{eqn:minimal}).

%%%%%%%%%%%%%%%%%%%%%%%%%%%%%%%
\subsection{}  We shall explain the relation between the two notions of constant cycle curves and state some
basic properties.

\begin{prop}\label{lem:ccc} Let $X$ be a K3 surface over an algebraically closed field $k$.

i) Let $C\subset X$ be a curve and let $K/k$ be an algebraically closed base field extension.
Then $C$ is a constant cycle curve of order $n$  if and only
if $C_K\subset X_K$ is a constant cycle curve of order $n$.

ii) Assume $X$ is not (Artin) supersingular and
${\rm char}(k)\ne2$.\footnote{This assumption holds whenever ${\rm char}(k)=0$ and
see below for a reminder on the notion of supersingular K3 surfaces.
What is really needed in the proof is $\rho(X)\ne22$ or  $\Ch^2(X_{k'})\ne\ZZ$ for some algebraically closed
extension $k'/k$. Thanks to Burt Totaro for pointing out how
to weaken the original assumption.}
If $K/k$ is an extension with $K$ also algebraically closed and $D\subset  X_K$ is a
constant cycle curve, then $D$ descends, i.e.\ there exists
a constant cycle curve $C\subset X$ with $D=C_K$.

iii)  If $X$ is defined over a (finitely generated) field $k_0$ with $\bar k_0=k$ and $C\subset X$ is a constant cycle curve,
then the natural $Gal(\bar k/k_0)$-action applied to $C$ yields only  constant cycle curves.
\end{prop}

\begin{proof} i) The pull-back
\begin{equation}\label{eqn:Lecomte}
\Ch^2(X\times_k k(\eta_C))\to \Ch^2(X_K\times_Kk(\eta_{C_K}))
\end{equation}
induced by the base change $X_K\times_Kk(\eta_{C_K})=(X\times_kk(\eta_C))\times_kK\to X\times_kk(\eta_C)$
has torsion kernel and maps $\kappa_C$ to $\kappa_{C_K}$. Hence, $C$ is a constant cycle curve if and only if
$C_K$ is.

That the order does not change can be shown using arguments of Lecomte
in \cite{Lecomte}, where it is proved that for any variety $Y$ over  an algebraically closed field $K_0$
 the base change $$\Ch^*(Y)\to\Ch^*(Y_K)$$ to a larger algebraically closed field $K/K_0$ induces an isomorphism on torsion.
(Injectivity is enough for our purpose, for which only $K_0$ algebraically closed is needed.) It cannot be applied directly,
as in our case $K_0=k(\eta_C)$ is not algebraically closed, but we may apply it to $Y=X\times U$ for 
open subsets $U\subset C$ to obtain
 $$\xymatrix@R-8pt{\Ch^2(X\times{k(\eta_C)})\ar[d]^\wr&&\Ch^2(X_{K}\times_K k(\eta_{C_K}))\ar[d]^\wr\\
 \lim\limits_{\rightarrow}\Ch^2(X\times_k U)\ar@{^{(}->}[r]&\lim\limits_\rightarrow\Ch^2(X_K\times_K U_K)\ar[r]&\lim\limits_\rightarrow\Ch^2(X_K\times_K V),
 }$$
 where the last map might a priori not be injective. Suppose that the pull-back $\alpha_K$ of $\alpha\in \Ch^2(X\times_k U)$ yields a
 trivial class in $\Ch^2(X_K\times_KV)$ for some open set $V\subset U_K$, which we can assume to be the complement
of finitely many closed points $p_1,\ldots,p_m \in U_K$. (One can further reduce to the case that each of the $p_i$ dominates $C$,
otherwise shrink $U$). Now use the  localization exact sequence (see \cite{Bloch,Fulton,VoisinHodge}):
$$\Ch^1(X_K\times\{p_1,\ldots,p_m\})\to\Ch^2(X_K\times_K U_K)\to \Ch^2(X_K\times_K V)\to 0$$
to conclude that $\alpha_K$ is supported on $\{p_1,\ldots,p_m\}$.
Write $$\Ch^2(X_K\times_KU_K)=\lim\Ch^2(X\times_kU\times_k W)$$ with the limit over all
non-empty open subsets $W\subset{\rm Spec}(A)$ for all
(finitely generated) $k$-algebras $k\subset A\subset K$ (cf.\ \cite{Lecomte}). For
small $W\subset {\rm Spec}(A)$ represent $\alpha_K$ by 
$\alpha_W\in\Ch^2(X\times_kU\times_kW)$. Its restriction $\alpha_t$ to a closed point $t\in W$
gives back $\alpha$. Indeed, since  $k$ is algebraically closed, one has
$k(t)\cong k$ and $\alpha\mapsto\alpha_W\mapsto\alpha_t$ is given by the isomorphism $\Ch^2(X\times_kU)\to \Ch^2(X\times_kU\times_k k(t))$.

But for small $W\subset{\rm Spec}(A)$ the class $\alpha_t$ is supported on the intersection of the closure of $\{p_1,\ldots,p_m\}$
in $X\times_kU\times_kW$ with the fibre over $t$, which is a finite set of points $\{p_{1t},\ldots,p_{mt}\}$ in $X\times_kU\times_kk(t)$. Hence,
$\alpha$ restricted the complement of these points is trivial. This eventually shows that $\alpha$ represents the trivial class
in $\Ch^2(X\times{k(\eta_C)})$.

ii) If $D\subset X_K$ is not defined over $k$, then there exists a one-dimensional
family of curves $C_t\subset X$ which can be seen as specializations of $D$. 
In particular, their classes $\kappa_{C_t}$ are obtained by specializing $\kappa_D$ and, therefore, are also torsion. Thus, one would obtain a dominant family
of pointwise constant cycle curves $C_t\subset X$. For every larger algebraically closed
$k'/k$ base changing the family $\{C_t\}$
defines a dominating family of constant cycle curves for $X_{k'}$ and which would imply $\Ch^2(X_{k'})\cong\ZZ$,
which is excluded for non-supersingular K3 surfaces, as e.g.\ $\Ch^2(X\times \overline{k(\eta_X)})\ne\ZZ$.
Recall that a K3 surface $X$ over a field $k$ of ${\rm char}(k)>0$ is called Artin supersingular if its height is infinite and
Shioda supersingular if $\rho(X_{\bar k})=22$. It has been known for a long time that
Shioda supersingular implies Artin supersingular and the proof of  the converse  has recently be completed
in \cite{Charles,Maulik, Keerthi}. (It has also been known that unirational K3 surfaces are Shioda (and hence Artin)
supersingular and the converse has been established in \cite{Liedtke}.)
Thus, a  K3 surface $X$ over $k$ with ${\rm char}(k)\ne2$
is supersingular if and only if $\rho(X)\ne22$, i.e.\ $H^2_{et}(X,\QQ_\ell(1))\ne{\rm NS}(X)\otimes\QQ_\ell$.
Hence, Bloch's result \cite[Thm.\ 6, Appendix to Sec.\ 1]{Bloch} applies.

iii) The last assertion is obvious, as  the notion of constant cycle curves is scheme-theoretic.
\end{proof}

For the reader's convenience and later use, we recall the following fact, which is a special case of a result due to Voisin (cf.\ \cite[Ch.\  22]{VoisinHodge}) improving upon a result of Bloch and Srinivas \cite{BlochSrini}.

\begin{prop}\label{prop:Voisin}
Assume $k$ is  algebraically closed. Then  a constant cycle curve
$C\subset X$ is also a pointwise constant cycle
curve. If $k$ is uncountable, the converse holds true as well.
\end{prop}

\begin{proof} Clearly, we may assume that $C$ is integral and for simplicity we also assume that $C$ is smooth
(otherwise pass to its normalization and replace $C\,\hookrightarrow X$ by the generically finite
map $\widetilde C\to X$).
The first assertion is easy. If $\kappa_C$ is torsion, say $n\cdot\kappa_C=0$, then there exists
an open subset $U:=C\setminus\{p_1,\ldots,p_m\}\subset C$ such that
$0=n\cdot[(\Delta_C-\{x_0\}\times C)|_{X\times U}]\in\Ch^2(X\times U)$ (use (\ref{eqn:dirlimi})). By 
the localization exact sequence 
\begin{equation}\label{eqn:loc}
\Ch^1(X\times\{p_1,\ldots,p_m\})\to\Ch^2(X\times C)\to \Ch^2(X\times U)\to 0,
\end{equation} we can assume that
$n\cdot(\Delta_C-\{x_0\}\times C)$ is rationally equivalent to a cycle $Z$ on $X\times C$ with support in $X\times\{p_1,\ldots,p_m\}$. 
As any zero-cycle on $C$ is linearly equivalent to one disjoint to the finite set of points $\{p_1,\ldots,p_m\}$,
the induced map $[Z]_*:\Ch_0(C)\to\Ch^2(X)$ is trivial. Thus, $n\cdot[\Delta_C-\{x_0\}\times C]_*:\Ch_0(C)
\to\Ch^2(X)$ is trivial and, since $\Ch^2(X)$ is torsion free, also $[\Delta_C-\{x_0\}\times C]_*=0$. The latter
is equivalent to saying that $[x]\equiv[x_0]=c_X$ for all closed points $x\in C$.

For the converse use \cite[Cor.\ 22.20]{VoisinHodge}, which is stated for $k=\CC$ but in fact holds for any uncountable
field $k$. Then $n\cdot[\Delta_C-\{x_0\}\times C]=[Z]$ for some $n>0$ with $Z$ supported
on a closed set of the form  $X\times \{p_1,\ldots,p_m\}$. But then $Z|_{X_{k(\eta_C)}}=0$ and hence $\kappa_C\in\Ch^2(X\times{k(\eta_C)})$
is torsion. \end{proof}

\subsection{} One could avoid mentioning the distinguished class $c_X\in\Ch^2(X$) in the definition of a constant
cycle curve altogether by proving analogously to Lemma \ref{lem:BVclass} the next

\begin{lem}\label{lem:BVclass2} Let $X$ be a K3 surface over an algebraically closed field $k$.
For an integral curve $C\subset X$ the following conditions are equivalent:

i) The curve $C$ is a constant cycle curve.

ii) There exists a positive integer $n$ such that 
\begin{equation}\label{eqn:minimal2}
n\cdot[\eta_C]\in {\rm Im}\left(\Ch^2(X)\to\Ch^2(X\times_k k(\eta_C))\right),
\end{equation} where the generic point
$\eta_C\in C$ is viewed as a closed point in $X\times_k{k(\eta_C)}$. 

iii) If  $\eta_C\in C$ is viewed as a point in
the geometric generic fibre $X\times_k{\overline {k(\eta_C)}}$, then
$$[\eta_C]\in {\rm Im}\left(\Ch^2(X)\to\Ch^2(X\times_k \overline{k(\eta_C)})\right).$$ 
\end{lem}

\begin{proof}
Clearly, i) implies ii) and iii). Since $\Ch^2(X\times_k k(\eta_C))\to\Ch^2(X\times_k \overline{k(\eta_C)})$
has torsion kernel and torsion free target, ii) and iii) are equivalent.

It remains to show that ii) implies i). Now, for a closed point $x\in C$,  the composition of the specialization map 
$s_x:\Ch^2(X\times k(\eta_C))\to \Ch^2(X)$ (see \cite[Ch.\ 20.3]{Fulton})  with the pull-back (\ref{eqn:minimal2})
yields the identity on $\Ch^2(X)$. If now $n\cdot [\eta_C]$ is in the image of
 (\ref{eqn:minimal2}), say $n\cdot [\eta_C]=\alpha_{k(\eta_C)}$, then $n\cdot[x]=\alpha$, as
 $\eta_C$ clearly specializes to $[x]$. Thus, $C$ is a pointwise constant cycle curve, which for an uncountable field
 is enough to conclude (use Proposition \ref{prop:Voisin}). 
 
 If $k$ is only countable, use  base change to an uncountable algebraically closed extension
 $K/k$ (e.g.\ a universal domain). Clearly, then $n\cdot [\eta_{C_K}]=n\cdot[\eta_C]_K$ is contained
 in the image of $\Ch^2(X)\to\Ch^2(X_K\times_K k(\eta_{C_K}))$. Hence, $C_K$ is a constant cycle curve and, by Lemma
 \ref{lem:ccc}, i), also $C$ is.
\end{proof}

%%%%%%%%%%%%%%%%%%%%%%%%%%%%%%%%%%%%%%%%%%%%%%%%%%%%%%%%%%%%%%
\subsection{}\label{McLeanunderbasechange}
Let $K/k$ be an extension of algebraically closed fields. Let $X$ be a K3 surface over $k$
and denote by $X_K$ the K3 surface over $K$ obtained by base change. The base change morphism
$\xi:X_K\to X$ induces  the pull-back map $\xi^*:\Ch^2(X)\,\hookrightarrow\Ch^2(X_K)$, i.e.\ $[x]_K=\xi^*[x]$ for all
closed points $x\in X$. It is  injective, but not surjective as soon as ${\rm trdeg}_k(K)\geq 2$ and $\rho(X)\ne22$ by \cite{Bloch} or
${\rm trdeg}_k(K)\geq 1$ and ${\rm char}(k)=0$ (see \cite{GGP}).

Compare the following also to \cite[Prop.\ 5]{Gorchinsky}.

\begin{cor} Let $x\in X_K$ be a closed point with $$[x]\in{\rm Im}\left(\Ch^2(X)\to\Ch^2(X_K)\right),$$
e.g.\ $[x]=c_{X_K}=(c_X)_K$.
Then one of the following is true:

i) The image $\xi(x)\in X$ is a closed point, i.e.\ $x$ is defined over $k$.

ii) The closure $C:=\overline{\{\xi(x)\}}\subset X$ of  $\xi(x)\in X$ is a constant cycle curve.

iii) The image $\xi(x)\in X$ is the generic point of $X$ and $\Ch^2(X)\cong\ZZ$.
\end{cor}

\begin{proof} Suppose $x$ is not defined over $k$. Then its image in $X$ is either
the generic point of $X$ or of a curve $C\subset X$. In the second case, 
consider the natural inclusion $k(\eta_C)\,\hookrightarrow k(x)\cong K$
for the generic point $\eta_C\in C$. The induced map $\Ch^2(X\times{k(\eta_C)})\to\Ch^2(X_K)$
sends $[\eta_C]$ to $[x]$. Since the kernel is torsion, one can conclude by Lemma \ref{lem:BVclass2}.
Similarly, if $x$ is mapped to $\eta_X\in X$, then $[\eta_X]$ is up  to torsion
contained in the image of $\Ch^2(X)\to\Ch^2(X\times_kk(\eta_X))$. And then, by specialization,
in fact $[y]\equiv{\rm const}$ for all points $y\in X$.
\end{proof}

So, in principle, one could try to produce constant cycle curves by finding points $x\in X_K$ with $[x]=c_{X_K}$ 
not defined over $k$. Then either $\Ch^2(X)\cong\ZZ$ or the closure of $x$ in $X$ is a constant cycle curve.
Although finding points $x$ not defined over $k$ is in principle possible, deciding whether also $[x]=c_{X_K}$ is difficult to verify without 
knowing beforehand that $x$ is contained in a constant cycle curve $C\subset X_K$ which for $\Ch^2(X)\not\cong\ZZ$ would automatically
descend to $k$, cf.\ Lemma \ref{lem:ccc}. Also note that
by i) in Lemma \ref{lem:ccc} one might expect that the order of the constant cycle curve 
$C:=\overline{\{\xi(x)\}}\subset X$ in situation ii) is an invariant of the point $x\in X_K$,
but how to read it off directly from $x$ is unclear.

\begin{cor}\label{cor:cccthroughpt}
Let $x\in X_K$ be a closed point not defined over $k$. Then 

i) Either, $[x]=(c_{X})_K\in\Ch^2(X_K)$ and  there exists a constant cycle curve  $x\in C\subset X_K$\\ or 

ii) the class $[x]$ is not contained in the image of $\Ch^2(X)\to\Ch^2(X_K)$.
\end{cor}
\begin{proof}
Indeed, take the curve $\overline{\{\xi(x)\}}\subset X$ (or any curve contained in it) and base change
it to a curve in $X_K$.
\end{proof}

Rephrasing this result yields another proof of a weak form of the main result of \cite{GGP} for K3 surfaces
and of the original result by Bloch \cite{Bloch}.
Contrary to the  proof in \cite{GGP}, the arguments here do not involve Hodge theory and, therefore, work
in positive characteristic. Roughly, the result says, if $\Ch^2(X)_0\ne0$, then $\Ch^2(X)$ grows
under any base field extension to a bigger algebraically closed field.

\begin{cor}\label{cor:BGGPweak}
Assume $K/k$ is an extension of algebraically closed fields with  ${\rm trdeg}_k(K)>0$.
If $X$ is a K3 surface over $k$ with $\Ch^2(X)\not\cong\ZZ$, then 
$$\Ch^2(X)\to\Ch^2(X_K)$$
is not surjective.\qed
\end{cor}

\begin{remark} The following is rather speculative and probably well-know to experts:
One may wonder, if the above opens  a way to prove the Bloch--Beilinson conjecture
for K3 surfaces (see Conjecture \ref{conj:BBC}). Suppose $X$ is a K3 surface over $k=\bar\QQ$.
If for any $x_0\in X(k)$ there exist a field extension $K/k$ and a point $x\in X(K)\setminus X(k)$ with $[x]=[x_0]_K$,
then $\Ch^2(X)\cong\ZZ$. Indeed, by Corollary \ref{cor:cccthroughpt} one would have
$[x]=(c_X)_K$,
as $[x]$ is by assumption contained in the image of
$\Ch^2(X)\to\Ch^2(X_K)$, and hence $[x_0]=c_X$.
Unfortunately, I do not know of any method that could possibly construct such a point $[x]$.
Note that it has to be important that the original surface is defined over a number field, i.e.\ $k=\bar\QQ$,
as we do not expect $\Ch^2(X)\cong\ZZ$ for $X$ over bigger fields. (A concrete counterexample
has been given by Schoen, see \cite{Jannsen}.)
\end{remark}

Similarly to (\ref{eqn:McLean}) one can define for any point $x_0\in X(k)$ the set 
\begin{equation}\label{eqn:McLean2}
X_{[x_0]}(k):=\{ x\in X(k)~|~[x]=[x_0]\in\Ch^2(X)\},
\end{equation}
which due to Maclean's result \cite{ML} is dense (at least for generic complex $X$).
If $C\subset X$ is a constant cycle curve, then $C(k)\subset X_{c_X}(k)$. Now consider
a  non-trivial extension $K/k$ with $K$ algebraically closed. Then  
$C_K\subset X_K$ remains a constant cycle curve. Clearly, not all points in $C_K$ will
be defined over $k$ and, therefore, the natural inclusion 
$$X_{c_X}(k)\subset X_{c_X}(K)$$
is strict. For any other class $[x_0]\ne c_X$
the set of points realizing it does not grow
under base field extension, as shown by the next result
which is again just a reformulation of Corollary \ref{cor:cccthroughpt}.

\begin{cor} Let $K/k$ be any extension of algebraically closed fields. If $X$ is a K3
surface over $k$ and $x_0\in X$ is a closed point with $[x_0]\ne c_X$, then
$$X_{[x_0]}(k)={X}_{[x_0]}(K),$$
i.e.\ all points in $X_K$ rationally equivalent to $x_0$ are defined over $k$.\qed 
\end{cor}

%%%%%%%%%%%%%%%%%%%%%%%%%
\section{Constant cycle curves on other surfaces}\label{sec:BlochConj}
The notion of constant cycle curves makes sense for other types of surfaces, see Section \ref{sec:defkappa}. We shall briefly discuss the case of
surfaces satisfying Bloch's conjecture. Recall that Bloch's conjecture for surfaces $X$ with $p_g(X)=0$ predicts that the kernel of
the Albanese map
$$\Ch^2(X)_0\to{\rm Alb}(X)$$ is trivial. It has been verified in \cite{BKL} for all surfaces of Kodaira dimension $\leq1$ and for many
surfaces of general type, in particular those dominated by products of curves. In \cite[Cor.\ 7.7]{Kimura} it has been shown that the finite-dimensionality of the Chow motive ${\mathfrak h}(X)$ (in the sense of Kimura--O'Sullivan)
implies Bloch's conjecture. For simplicity we will restrict to the case  $q(X)=0$,
otherwise one would have to restrict to curves in the fibres of the Albanese map.

\begin{prop} Let $X$ be a smooth projective surface with $p_g(X)=q(X)=0$ over an algebraically closed field $k$ of characteristic zero.
Then $X$ satisfies Bloch's conjecture, i.e.\ $\Ch^2(X)_0=0$, if and only if every curve in $X$ is a constant cycle curve.
\end{prop}

\begin{proof} Suppose $X$ satisfies Bloch's conjecture.
If $k$ is uncountable, e.g.\ $k=\CC$, then the assertion follows from the arguments used to prove Proposition \ref{prop:Voisin}. Indeed,
assuming Bloch's conjecture, every curve is a pointwise constant cycle curve. For arbitrary field $k$ one
uses \cite[Thm.\ 7]{GP} (or directly  the techniques of
\cite{BlochSrini}) which states that Bloch's conjecture is equivalent to the finite-dimensionality of ${\mathfrak h}(X)$.
By \cite{KMP} the Chow motive ${\mathfrak h}(X)$ is finite-dimensional if and only if its transcendental motive
${\mathfrak t}^2(X)$ is finite-dimensional. But for $p_g(X)=0$, the latter is equivalent to ${\mathfrak t}^2(X)=0$ and by \cite[Cor.\ 7.4.9]{KMP}
this implies $\Ch^2(X\times k(\eta_X))_0\otimes\QQ=0$ (see also the discussion in \cite[Ch.\ 4.1]{Andre}). To conclude, use that the
generic point of any curve $C\subset X$ is a specialization of $\eta_X$.
Conversely, if every curve $C\subset X$ is a constant cycle curve,
then clearly $\Ch^2(X)\cong\ZZ$. \end{proof}

However, determining the order of  constant cycle curves in these surfaces is a different matter. Mainly, because nothing
seems to be known about the integral version of 
the transcendental motive ${\mathfrak t}^2(X)$ and the torsion in $\Ch^2(X\times k(\eta_X))$ is difficult
to control.

The following cases are instructive for our purpose.

\begin{prop}\label{prop:ccconrat}
Let $X$ be a  surface over an algebraically closed field $k$ with $p_g(X)=q(X)=0$ satisfying Bloch's conjecture $\Ch^2(X_K)_0=0$ for all algebraically closed extensions $K/k$.

i) There exists an integer $n$ such that all curves on $X$ are constant cycle curves
of order $\leq n$.\footnote{Thanks to Claire Voisin for suggesting this.}

ii) If $X$ is rational, then every curve $C\subset X$ is a constant cycle curve of order one.

iii) If $X$ is an Enriques surface, then all curves are constant cycle curves
of order $n| 4$.
\end{prop}

\begin{proof} 
i) For an integral curve $C\subset X$ consider the specialization map
$$\Ch^2(X\times k(\eta_X))\to \Ch^2(X\times k(\eta_C))$$ which sends $[\eta_X]$ to $[\eta_C]$. Thus,
it suffices to show that there exists an $n$ with $n\cdot[\eta_X]=n\cdot [x_0\times k(\eta_X)]$
for some $x_0\in X$ or, equivalently, that $[\eta_X]-[x_0\times k(\eta_X)]$ is contained in the kernel
of $\Ch^2(X\times k(\eta_X))\to\Ch^2(X\times\overline{k(\eta_X)})$. But by assumption $\Ch^2(X\times\overline{k(\eta_X)})\cong\ZZ$ and $[\eta_X]-[x_0\times k(\eta_X)]$ is of degree zero.
\smallskip

ii) As rational curves are constant cycle curves of order one
(see Lemma \ref{sec:ratcurves} for the argument which is a direct consequence
of the definition), we may also replace  $X$ by a  minimal model.

So, if $X$ is rational, we can in fact assume  $X\cong\PP^2$. Now use that $[\Delta_{\PP^2}]\in\Ch^2(\PP^2\times\PP^2)$ can be written as
$[\Delta_{\PP^2}]=h_1^2\times 1+h_1\times h_2+1\times h_2^2$, where $h_i$, $i=1,2$,
denote the hyperplane sections on the two factors. Hence, for every integral curve $C\subset\PP^2$ one finds
$[\Delta_{C}]=[\Delta_{\PP^2}]|_{\PP^2\times C}=h_1^2\times[C]+h_1\times h_2|_C$ and, therefore,
$\kappa_C=0$.
\smallskip

iii)
Every Enriques surface admits an elliptic fibration $X\to \PP^1$ with a $2$-section $C_0\subset X$.
Then the base change to $X\times k(\eta_C)$ for any curve $C\subset X$ comes with a natural $2$-section, too. Now copy the arguments in  \cite{BKL} to show that every class in $\Ch^2(X')_0$ for an Enriques surface
$X'$ over an arbitrary field $k'$ is annihilated by $n^2$, when $n$ is the degree of a multi-section. In \cite{BKL}
this was combined with the fact that $\Ch^2(X')_0$ is torsion free for $k'$ algebraically closed
to deduce $\Ch^2(X')_0=0$. Here, we apply it to $X'=X\times k(\eta_C)$ and $k'=k(\eta_C)$ to
conclude the assertion.
\end{proof}

%%%%%%%%%%%%%%%
\section{Finiteness of constant cycle curves of fixed order}\label{sec:finite}
The aim of the section is to prove

\begin{prop}\label{prop:finite} Let $X$ be a projective K3 surface over an algebraically closed field of charac\-teris\-tic zero.
Then there are at most finitely many constant cycle  curves $C$ of fixed order $n$ in any linear system
$|L|$.
\end{prop}

It is clearly enough to prove the theorem for complex K3 surfaces. The techniques, involving Hodge theory and normal functions,
do not apply to K3 surfaces over fields of positive characteristic. In fact, for unirational K3 surfaces over $\bar\FF_p$,
which come with infinitely many rational and hence constant cycle curves of order one in a certain linear system, the result is clearly false.
Whether this is the only exception is not clear.

%%%%%%%%%%%%%%%%%%%%%%%%

\subsection{}\label{sec:ZC} 
For the Hodge theoretic considerations below, we first need to introduce a compactification of $\kappa_C$  to a class
in $\Ch^2(X\times C)$ that is different from the naive one used e.g.\ in the proof of Proposition \ref{prop:Voisin}.

Let $X$ be a K3 surface with a fixed point $x_0\in X$ such that $[x_0]=c_X$. Consider an integral curve $C\subset X$ and its normalization $\widetilde C\to C$. As before,
the composition
shall be denoted $f_C:\widetilde C\to C\subset X$. Given $C$ and
cycle $\sum n_i \cdot y_i$ of degree one on $\widetilde C$,
one defines 
\begin{equation}\label{eqn:defZC}
Z_C:=\Delta_{f_C}-C\times\left(\sum n_i\cdot y_i\right)-\{x_0\}\times \widetilde C
\end{equation}
which is an integral cycle on the smooth threefold $X\times \widetilde C$. Here,
$\Delta_{f_C}$ denotes the graph of $f_C$ which can also be seen as the pull-back of the
diagonal $\Delta_X\subset X\times X$ under $({\rm id},f_C):X\times\widetilde C\to X\times X$.
Note that, although not reflected by the notation,
$Z_C$ and the associated class $[Z_C]\in\Ch^2(X\times\widetilde C)$ do
depend on the cycle $\sum n_i\cdot y_i$ resp.\ the associated line bundle $L_0:=\ko\left(\sum n_i\cdot y_i\right)$ on $\widetilde C$.
(In Lemma \ref{lem:trAJ} below the transcendental Abel--Jacobi class associated with $Z_C$ will be shown to be independent of $\sum n_i\cdot y_i$.)
 Geometrically, it
would seem natural to choose $L_0$ to be of the form $\ko(y_0)$ for some point
$y_0\in \widetilde C$. However, for arguments involving families of curves $C$, allowing line bundles
has technical advantages, cf.\ Section \ref{sec:normal}. In fact, we will later choose $L_0$ such that $L_0^{2g-2}\cong f_C^*L$,
for $C\in|L|$ with $L.L=2g-2$.

It is straightforward to check that $Z_C$ is homologously trivial, e.g.\ for a complex K3 surface
the induced map $$[Z_C]_*:H^*(\widetilde C,\ZZ)\to H^*(X,\ZZ)$$ is zero.
This allows one to define its Abel--Jacobi class, see Section \ref{subsec:AJ}.

For an integral curve $C$, the generic points $\eta_{\widetilde C}\in\widetilde C$ and
$\eta_C\in C$ can be identified. Then

\begin{lem} For an integral curve $C\subset X$,  restriction to the generic fibre maps
$[Z_C]$ to $\kappa_C$, i.e.
$$\Ch^2(X\times \widetilde C)\to\Ch^2(X\times{k(\eta_C)}),~~[Z_C]\mapsto\kappa_C.$$
\end{lem} 
Clearly, if $[Z_C]$ is torsion, then so is $\kappa_C$ and hence $C$ is a constant cycle curve, but the converse is not
true. If in addition  $L_0^{2g-2}\cong f_C^*L$ is required, then at least
the question whether the cycle $Z_C$ is torsion does no
longer depend on the particular $L_0$, for another choice of $L_0$ differs by torsion in $\Pic^0(\widetilde C)$.

\begin{remark}\label{rem:comp}
We briefly discuss various other possible choices for $Z_C$. We restrict to the case of a complex K3 surface $X$.

i) On $X\times X$ one often
considers $$Z_X:=\Delta_X-\sum n_i \cdot(C_i\times D_i)-\{x_0\}\times X-X\times\{x_0\},$$ the
transcendental part of the diagonal. The curves
$C_i,D_i\subset X$ and the $n_i\in \QQ$ are chosen such that $[Z]_*$ is trivial on $H^0(X,\QQ)\oplus H^4(X,\QQ)\oplus {\rm NS}(X)\otimes\QQ\subset H^*(X,\QQ)$, i.e.\ $[Z_X]\in T(X)\otimes T(X)\otimes\QQ\subset H^4(X\times X,\QQ)$, where $T(X)\subset H^2(X,\ZZ)$
is the transcendental lattice. For $\rho(X)=1$, this can be rewritten as
\begin{equation}\label{eqn:ZX}
Z_X=\Delta_X-(1/(2g-2))\cdot(C\times C)-\{x_0\}\times X-X\times\{x_0\},
\end{equation}
with $C.C=2g-2$, which is defined in general and satisfies $(2g-2)\cdot[Z_X]\in [C]^\perp\otimes[C]^\perp\subset H^2(X,\ZZ)\times
H^2(X,\ZZ)$.

The pullback to $X\times\widetilde C$ would be another
natural choice for $Z_C$ with the property that its restriction to $X\times{k(\eta_C)}$ represents $\kappa_C$.
The obvious advantage to be defined universally on $X\times X$ has the price, due to the coefficient
$1/(2g-2)$, of being only rationally
defined, i.e.\ $[({\rm id},f_C)^*Z_X]$ is well defined only in $\Ch^2(X\times\widetilde C)\otimes\QQ$, which makes it more difficult to work with its Abel--Jacobi class. 

 Similarly, the class defined by the cycle
$$Z'_C:=\Delta_{f_C}-(1/(2g(\widetilde C)-2))\cdot(C\times D)-\{x_0\}\times C$$
with $D\in|\omega_C|$ would be independent
of any additional choice of $L_0\in{\rm Pic}^1(\widetilde C)$, but again it is only rational. 

\smallskip

ii) In \cite{GG} Green and Griffiths study a cycle introduced by Faber and Pandharipande. For a divisor
$D$ of degree $d$ on a smooth curve $C$ they define $$z_D:=D\times D-d \cdot D_\Delta.$$ Assume 
now that $C$ is contained in a K3 surface $X$ and suppose for simplicity that $C$ generates ${\rm NS}(X)$. Then
for $D\in|\omega_C|$ 
$$z_D\sim(2-2g(C))\cdot Z_X|_{C\times C}$$ 
under $C\times C\subset X\times X$,
where in this case  $Z_X$ is as in (\ref{eqn:ZX}).
Similarly, $z_D$ is rationally equivalent to the restriction of $(2-2g(C))\cdot Z_C'$ under
$C\times C\subset X\times C$.

As pointed out in \cite{GG}, the class $[z_D]\in\Ch^2(C\times C)\otimes\QQ$ for $D\in|\omega_C|$ is (by Bloch--Beilinson) conjectured to be trivial for curves $C$
defined over $\bar\QQ$, but it is also shown that it is not trivial for general $C$ of genus $g(C)\geq4$
(see also the shorter proof in \cite{Yin}). For $4\leq g(C)<10$ this
can be used to show that for general $C\subset X$ (with varying $X$) the class
$[Z'_C]\in\Ch^2(X\times C)\otimes\QQ$ is non-trivial.
\end{remark}

\subsection{}\label{subsec:AJ} Consider the Abel--Jacobi map  for the smooth threefold $X\times\widetilde C$:
$${\rm AJ}:\Ch^2(X\times\widetilde C)_{\rm hom}\to J^3(X\times \widetilde C), ~~\gamma\mapsto\int_\Gamma.$$
Here, $\Ch^2(X\times\widetilde C)_{\rm hom}$ denotes the homologically trivial part
and for any $\gamma$ we let $\Gamma$ be
a real three-dimensional cycle with $\partial\Gamma=\gamma$.
The intermediate Jacobian
$$J^3(X\times \widetilde C):=\frac{H^3(X\times \widetilde C,\CC)}{F^2H^3(X\times\widetilde C)+H^3(X\times \widetilde C,\ZZ)}
\cong \frac{F^2H^3(X\times \widetilde C)^*}{H^3(X\times\widetilde C,\ZZ)}$$
is a complex torus of dimension $22\cdot g(\widetilde C)$. 

We shall also need the `transcendental part' of the intermediate Jacobian which we  define as
(see also \cite{Green,VoisinHigher}): 
\begin{equation}\label{eqn:Jactr}
J^3(X\times\widetilde C)_{\rm tr}:=\frac{F^2(T(X)\otimes H^1(\widetilde C,\CC))^*}{T(X)\otimes H^1(\widetilde C,\ZZ)}.
\end{equation} 
Here, $T(X):={\rm NS}(X)^\perp\subset H^2(X,\ZZ)$ is  the transcendental lattice which equals
the orthogonal complement $[C]^\perp\subset H^2(X,\ZZ)$ if  $\rho(X)=1$. We shall denote the composition of
the projection $J^3(X\times\widetilde C)\to J^3(X\times\widetilde C)_{\rm tr}$ with the
Abel--Jacobi map by
$${\rm AJ}_{\rm tr}:\Ch^2(X\times \widetilde C)_{\rm hom}\to J^3(X\times\widetilde C)_{\rm tr}.$$

The following is well-known, see e.g.\ \cite{VoisinHigher}:
\begin{lem}\label{lem:trAJ}
The transcendental Abel--Jacobi map ${\rm AJ}_{\rm tr}$ factorizes naturally
$${\rm AJ}_{\rm tr}:\Ch^2(X\times\widetilde C)_{\rm hom}/({\rm Pic}(X)\times\Ch^1(\widetilde C)_{\rm hom})\to J^3(X\times
\widetilde C)_{\rm tr}.$$
In particular, ${\rm AJ}_{\rm tr}(Z_C)$ is independent of the choice of $L_0=\ko(\sum n_i\cdot y_i)\in{\rm Pic}^1(\widetilde C)$
as in (\ref{eqn:defZC}).
\end{lem}

\begin{proof} One has to show that $\int_\Gamma$ is trivial on $T(X)\otimes H^1(\widetilde C,\CC)$
whenever $\partial\Gamma$ is of the form $D\times \sum n_i\cdot x_i$ for some curve
$D\subset X$ and with $x_i\in \widetilde C$ and $\sum n_i=0$. But in this case one can assume that $\Gamma$ is of the form $D\times \Gamma_0$ with a path
$\Gamma_0\subset \widetilde C$ and clearly $\int_D=0$ on $T(X)$.
\end{proof}

The Abel--Jacobi class of ${\rm AJ}(\kappa_C)$ is the class $e_{X,C}$ in \cite{Green}, which can also be understood as an extension class
in the category of mixed Hodge structures. The transcendental part is
(essentially) the class studied further in \cite{VoisinHigher}.
The second assertion of the following has been observed already by Green in \cite{Green}.

\begin{cor}\label{cor:kappaAJ}
Suppose $C\subset X$ is an integral curve. Then 
\begin{equation}{\rm AJ}_{\rm tr}(\kappa_C):={\rm AJ}_{\rm tr}(Z_C)\in
J^3(X\times \widetilde C)_{\rm tr}
\end{equation} is well-defined. If $C$ is a constant cycle curve of order $n$,
then $n\cdot{\rm AJ}_{\rm tr}(\kappa_C)=0$.
\end{cor}

\begin{proof} One argues as in the proof of Proposition \ref{prop:Voisin}.
Using (\ref{eqn:loc}) and Lemma \ref{lem:trAJ}, one finds that ${\rm AJ}_{\rm tr}(Z_C)$ only
depends on its restriction to some open subset $U:=\widetilde C\setminus \{p_1,\ldots,p_k)$.
But if $n\cdot\kappa_C=0$, then $U$ can be chosen such that $n\cdot[Z_C|_{X\times U}]=0$ in $\Ch^2(X\times U)$ and hence $n\cdot{\rm AJ}_{\rm tr}(Z_C)=0$.
\end{proof}

One can also project ${\rm AJ}(Z_C)$ onto a class
${\rm AJ}_{\rm alg}(Z_C)$ in the algebraic part of the intermediate Jacobian.
\begin{equation}\label{eqn:Jacalg} J^3(X\times \widetilde C)_{\rm alg}:=\frac{F^2({\rm NS}(X)\otimes H^1(\widetilde C,\CC))^*}{{\rm NS}(X)\otimes H^1(\widetilde C,\ZZ)}\cong\Pic^0(\widetilde C)^{\rho(X)}.
\end{equation}
To simplify the notation we shall henceforth
assume $\rho(X)=1$ (but see Remark \ref{rem:rho>1}), so that
$J^3(X\times \widetilde C)_{\rm alg}=F^2([C]\otimes H^1(\widetilde C,\CC))^*/([C]\otimes H^1(\widetilde C,\ZZ))\cong\Pic^0(\widetilde C)$.

\begin{lem}\label{lem:algpart} For an integral curve $C\in |L|$ with $L.L=2g-2$, one has
$${\rm AJ}_{\rm alg}(Z_C)=f_C^*L\otimes (L_0^{2g-2})^*\in\Pic^0(\widetilde C).$$
Thus, if $L_0$ is chosen such that $L_0^{2g-2}\cong f_C^*L$, then ${\rm AJ}_{\rm alg}(Z_C)=0$. Cf.\ \cite[p.\ 270]{Green}.
\end{lem}

\begin{proof} This must be standard. As I was not able to find a reference, I will sketch the argument.
Let us first explain how ${\rm AJ}_{\rm alg}(Z_C)$ changes with $L_0$. The difference of the cycles for two
different $L_0$ and $L_0'$ is a cycle of the form $C\times D$ with $D$ a degree zero cycle on $C$.
Then  for any one-form $\alpha$ on $C$ one finds
${\rm AJ}_{\rm alg}(Z_C-Z_C')(C\times \alpha)=\int_{C\times\delta}[C]\times \alpha=(2g-2)\int_\delta \alpha$, where $\partial\delta=D$. Hence, as predicted by the assertion, ${\rm AJ}_{\rm alg}(Z_C-Z_C')=\ko((2g-2)\cdot D)\cong
(L_0'\otimes L_0^*)^{2g-2}$. In particular, this allows one to restrict to the case $L_0=\ko(y_0)$ for some
point $y_0\in C$.

Now let $\gamma\subset X\times C$ with $\partial \gamma=Z_C$. Then $\int_\gamma [C]\times \alpha=
\int_{\gamma\cap(C'\times C)}{\rm pr}_2^*\alpha$, where the deformation $C'$ of $C$ is chosen generic
such that $\gamma\cap(C'\times C)$ is indeed one-dimensional. Note that $C'\cap C:=\{x_1,\ldots,x_{2g-2}\}$
is a divisor in $L|_C$. The intersection of $C'\times C$ with $Z_C$ 
consists of the points $(x_i,x_i)$ and (with negative sign) $(x_i,y_0)$ and the intersection of 
$C'\times C$ with $\gamma$ 
consists of  paths connecting these points. Now project onto the second factor, which
gives paths connecting the points $x_i={\rm pr}_2(x_i,x_i)$ with $y_0={\rm pr}_2(x_i,y_0)$, i.e.\ a path with boundary $L|_C\otimes \ko((2g-2)\cdot y_0)^*$.
(Alternatively, one can pull back $\alpha$.)\end{proof}

The natural inclusion $T(X)\oplus{\rm NS}(X)\subset H^2(X,\ZZ)$ is of finite index. Hence
\begin{equation}\label{eqn:Jactralg}
J^3(X\times C)\to J^3(X\times\widetilde C)_{\rm tr}\times J^3(X\times\widetilde C)_{\rm alg}
\end{equation}
is finite of degree say $N$ (which only depends on $[C]\in H^2(X,\ZZ)$ and not on $C$). In particular,
if both, ${\rm AJ}(Z_C)_{\rm tr}\in J^3(X\times\widetilde C)_{\rm tr}$ and ${\rm AJ}(Z_C)_{\rm alg}\in J^3(X\times\widetilde C)_{\rm alg}$ vanish, then also $N\cdot{\rm AJ}(Z_C)\in J^3(X\times\widetilde C)$.
%%%%%%%%%%%%%%%%%%

\subsection{}\label{sec:normal}

It is difficult to decide for any given curve $C\subset X$ whether ${\rm AJ}_{\rm tr}(\kappa_C)$ is non-trivial. 
However, when put in a family the resulting normal function is easier to control
and from its non-vanishing one deduces  the non-vanishing of ${\rm AJ}_{\rm tr}(\kappa_C)$ 
at least for generic $C$.
This type of argument is  standard (for hyperplane sections), but there are technical details that have
to be adjusted to our situation. The key point is eventually the algebraicity of the zero-locus of normal functions recently
established by Brosnan--Pearlstein \cite{BP} and M.\ Saito \cite{Saito}, see \cite{CharlesBourb} for a recent overview.

\medskip

We start with a positive-dimensional  family of curves $ T\subset|L|$   in a fixed linear system $|L|$ on $X$
with $L.L=2g-2$.
We assume that  all curves $C\in T_0$ parametrized by a dense open subset $T_0\subset T$
are integral with (analytically) constant  singularity type. 
Under this assumption, simultaneous  normalization
and  then compactification (for both, one may need to replace
the inclusion $T_0\subset |L|$ by a generically finite map
$T_0\to |L|$) yields a diagram
$$\xymatrix{f:\widetilde \kc|_{T_0}\ar[rd]_{p_0}\ar[r]&\kc|_{T_0}\ar[d]\ar[r]&X\\
&T_0&&}$$
with $f$ dominant and such that $p_0:\widetilde\kc|_{T_0}\to T_0$ is smooth with fibres given by
the normalizations of the curves
$\kc_t$,  $t\in T_0$. 
So, $f_t:\widetilde \kc_t\to\kc_t\subset X$ for $t\in T_0$ is exactly as in the situation $f_C:\widetilde C\to C\subset X$ considered before.

Recall that the class of the  cycle $Z_C$ on $X\times \widetilde C$ defined in (\ref{eqn:defZC}) depends on the choice of a line bundle
$L_0$ of degree one on $\widetilde C$. In order to define a global cycle that restricts to $Z_C$ on $X\times \widetilde C$
for a fibre $C=\kc_t$, $t\in T_0$, we pull-back $\kc|_{T_0}\to T_0$ to the subscheme $T_0'\subset\Pic^1(\widetilde\kc/T_0)$
of all line bundles $L_0$ on fibres $\widetilde C$ of
$\widetilde\kc|_{T_0}\to T_0$ such that $L_0^{2g-2}\cong f_C^*L$. 
Replacing $T_0$ by (some further \'etale cover of) $T_0'$,
we may assume that  $\widetilde\kc|_{T_0}\to T_0$ comes with  a line bundle $\kl_0$ on $\widetilde\kc|_{T_0}$
of degree one on all fibres $\widetilde C$ and  such that $\kl_0^{2g-2}|_{\widetilde C}\cong f_C^*L$.
Once this is achieved, one can compactify $\widetilde\kc|_{T_0}\to\kc|_{T_0}\to T_0$ to  projective
families $p:\widetilde\kc\to\kc\to T$ which still come with a morphism $f:\widetilde\kc\to\kc\to X$.

\smallskip

Next, consider the closure of the cycle 
\begin{equation}
Z_\kc:=\Delta_f-\kc|_{T_0}\times_T[\kl_0]-\{x_0\}\times\widetilde \kc|_{T_0}
\end{equation}
on $X\times\widetilde\kc$.  Here, $\Delta_f$ is the pull-back of $\Delta_X\subset X\times X$ under
${\rm id}\times f:X\times\widetilde \kc\to X\times X$. Restricted to $X\times \widetilde C$ for $C=\kc_t$, $t\in T_0$,
the cycle $Z_\kc$ yields $Z_C=\Delta_{f_C}-C\times[L_0]-\{x_0\}\times \widetilde C$ with $L_0\cong\kl_0|_{\widetilde C}$.

\medskip

We shall be interested in the normal function associated to this cycle.
Denote by $$J^3:=J^3(X\times\widetilde \kc/T_0)\to T_0$$ the family of intermediate Jacobians of
the family $$\pi:=p_0\circ pr_2:X\times\widetilde\kc|_{T_0}\to T_0$$ with fibres $X\times\widetilde\kc_t$.
Analogously, $J^3_{\rm tr}$ and $J^3_{\rm alg}$ denote the transcendental resp.\ algebraic parts
with fibres over $t\in T_0$ as described by (\ref{eqn:Jactr}) resp.\ (\ref{eqn:Jacalg}).
In particular, $$J^3_{\rm alg}\cong{\rm Pic}^0(\widetilde \kc/T_0)\to T_0.$$
The sheaf of sections of $J^3$, denoted by the same symbol, is part of the short exact sequence
\begin{equation}\label{eqn:ses}
0\to R^3\pi_*\ZZ\to \kh^3/F^2\kh^3\to J^3\to 0,
\end{equation}
where $\kh^3:=R^3\pi_*\Omega^\bullet_{X\times\widetilde\kc/T_0}$ and $F^2\kh^3:=R^3\pi_*\Omega^{\geq2}_{X\times\widetilde\kc/T_0}$.

Now, the fibrewise Abel--Jacobi classes ${\rm AJ}(Z_{\kc_t})$ induced by the global
cycle $Z_\kc$ yield the normal function $$\nu\in\Gamma(T_0,J^3),$$
i.e.\ a holomorphic section of $J^3\to T_0$, cf.\  \cite{VoisinHodge}. Similarly, its projections under
$J^3\to J^3_{\rm tr}$ and $J^3\to J^3_{\rm alg}$ are denoted $$\nu_{\rm tr}\in \Gamma(T_0,J^3_{\rm tr})
\text{ resp. }\nu_{\rm alg}\in\Gamma(T_0,J^3_{\rm alg}).$$

\begin{cor}\label{cor:trN}
Under the above choice of $Z_\kc$ and assuming $\rho(X)=1$, the algebraic part $$\nu_{\rm alg}\in\Gamma(T_0,J^3_{\rm alg})=\Gamma(T_0,{\rm Pic}^0(\widetilde\kc/T_0))$$ is trivial.
Moreover, for the zero sets of the normal functions and any $n>0$ one has
$$Z(n\cdot\nu_{\rm tr})\subset Z(n\cdot N\cdot \nu),$$ where $N$ is the degree of (\ref{eqn:Jactralg}).
\end{cor}

\begin{proof} The first part is an immediate consequence of Lemma \ref{lem:algpart}. For the second, recall that $n\cdot{\rm AJ}_{\rm tr}(Z_C)=0$ and $n\cdot{\rm AJ}_{\rm alg}(Z_C)=0$ imply $N\cdot n\cdot {\rm AJ}(Z_C)=0$.
\end{proof}

Clearly,  for $t\in T_0$ not contained in the zero locus of
$Z(N\cdot \nu)$ the Abel--Jacobi class $N\cdot{\rm AJ}(Z_C)$ and hence $N\cdot {\rm AJ}_{\rm tr}(Z_C)$ of $C:=\kc_t$ are non-trivial.
In order to  prove $N\cdot \nu\ne0$ and thus $T_0\setminus Z(N\cdot\nu)\ne\emptyset$, one describes its image under the boundary 
map  of (\ref{eqn:ses})
$$\delta:\Gamma(T_0,J^3)\to H^1(T_0,R^3\pi_*\ZZ)$$
and shows that it is actually non-trivial. Note that $R^3\pi_*\ZZ= H^2(X,\ZZ)\otimes R^1p_{0*}\ZZ$.

\begin{lem}\label{lem:boundary}
The boundary class $\delta(\nu)\in H^1(T_0,R^3\pi_*\ZZ)$ is non-torsion.
\end{lem}

\begin{proof} We can restrict to the case $\dim(T)=1$.

-- One first proves that $[Z_\kc]$ is a non-trivial class in $[C]^\perp\otimes H^2(\widetilde\kc|_{T_0},\ZZ)\subset H^4(X\times\widetilde\kc|_{T_0},\ZZ)$.

Note that the cohomology class $(2g-2)\cdot [Z_\kc]\in H^2(X,\ZZ)\otimes H^2(\widetilde\kc,\ZZ)$ is the
pull-back of $(2g-2)\cdot[\Delta_X]-[C]\times[C]\in H^2(X,\ZZ)\otimes H^2(X,\ZZ)$,
which is in fact the non-trivial class  $(2g-2)\cdot{\rm id}\in[C]^\perp\otimes [C]^\perp$ (cf.\ Remark \ref{rem:comp}).
(The argument is easily adapted to the case $g=1$.)
Since $\widetilde\kc\to X$ is dominant, the pull-back $T(X)\subset[C]^\perp\,\hookrightarrow H^2(\widetilde\kc,\ZZ)$ is injective.
The kernel of $H^2(\widetilde\kc)\to H^2(\widetilde\kc|_{T_0})$ is spanned by the divisor classes of the boundary
$\widetilde\kc\setminus\widetilde\kc|_{T_0}$ and, as $T(X)$ is an irreducible Hodge structure
of level two,  intersects $T(X)$ trivially. 

\smallskip
-- The further restriction to the fibres $\widetilde\kc_t$, $t\in T_0$, is trivial, i.e.\ under
$$[C]^\perp\otimes H^2(\widetilde\kc|_{T_0},\ZZ)\subset H^4(X\times\widetilde\kc|_{T_0},\ZZ)\to H^0(T_0,R^4\pi_*\ZZ)$$
the class $[Z_\kc]$ vanishes. This is just rephrasing that the $Z_C$ are cohomologically trivial cycles.

-- The Leray spectral sequence for $\pi:X\times\widetilde\kc|_{T_0}\to T_0$ yields
a natural (surjective) map
$$d:{\rm Ker}\left(H^4(X\times\widetilde\kc|_{T_0},\ZZ)\to H^0(T_0,R^4\pi_*\ZZ)\right)\to H^1(T_0,R^3\pi_*\ZZ),$$
the kernel of which is a quotient of the trivial $$H^2(T_0,R^2\pi_*\ZZ)=\left(H^2(X,\ZZ)\otimes H^2(T_0,p_{0*}\ZZ)\right)
\oplus \left(H^0(X,\ZZ)\otimes H^2(T_0,R^2p_{0*}\ZZ)\right)=0.$$
Here one uses that for dimension reasons all $H^i(T_0,R^j\pi_*\ZZ)$ are trivial for $i>2$ and also
$H^2(T_0,p_{0*}\ZZ)\cong H^2(T_0,R^2p_{0*}\ZZ)=0$ (after shrinking $T_0$, if necessary).

This yields a non-trivial  class $$d[Z_\kc]\in H^1(T_0,R^3\pi_*\ZZ)$$ as the image of $[Z_\kc]$.
Since all the arguments above apply as well to multiples
$N\cdot Z_\kc$, the class $d[Z_\kc]$ is in fact non-torsion.

-- Eventually, one uses the fact that the class  $d[Z_\kc]$ is indeed $\delta(\nu_Z)$. See \cite{VoisinHodge}.
\end{proof}

\begin{cor}\label{cor:finite}
Under the above assumptions, there exist at most finitely many points $t\in T_0$ with
${\rm AJ}_{\rm tr}(Z_{\kc_t})=0$.
\end{cor}

\begin{proof}
Here one uses that the vanishing locus $Z(N\cdot \nu)\subset T_0$ of the normal function $N\cdot\nu\in \Gamma(T_0,J^3)$ is 
an algebraic set due to \cite{BP,Saito}. Thus, if $Z(N\cdot\nu)$ is zero-dimensional, then it can only be a finite set of points and hence also
$Z(\nu_{\rm tr})$ is finite by Corollary \ref{cor:trN}. If $Z(N\cdot\nu)$ has positive
dimensional components, then repeat the above argument with $T_0$ replaced by such a component.  But over the new $T_0$
the class $N\cdot \delta [Z_\kc]$ would be trivial, contradicting that $\delta[Z_\kc]$ is non-torsion by Lemma \ref{lem:boundary}. 
\end{proof}

\begin{remark}\label{rem:rho>1}
For simplicity we assumed $\rho(X)=1$ in the above discussion. The case $\rho(X)>1$ is dealt with similarly, either
by replacing $H^2(X,\ZZ)$ throughout by $T(X)\oplus\ZZ\cdot[C]$ or by observing that the arguments above
prove directly that $\nu_{\rm tr}\ne0$ (without controlling the algebraic part) or by working with the cycle $Z_X$ in
Remark \ref{rem:comp}, i).
\end{remark}

\subsection{}
To conclude the proof of Proposition \ref{prop:finite}, one stratifies $|L|$ according to the singularity types of the curve and their number of integral components. Since eventually there are only finitely many strata and for each stratum finiteness of curves
$C:=\kc_t$ with $n\cdot {\rm AJ}_{\rm tr}(\kappa_C)=0$ is assured by Corollary \ref{cor:finite},  the
proposition then follows from Corollary \ref{cor:kappaAJ}.
%%%%%%%%%%%%%%%

\subsection{}\label{sec:disc}
 Suppose $\kc\subset\kx\to T$ is a flat family of curves $\kc_t$ in K3 surfaces $\kx_t$. Standard arguments prove that 
the locus of $t\in T$ for which $\kc_t\subset\kx_t$ is a constant cycle curve is a countable union of Zariski closed subsets of $T$.
Indeed, by using that the (relative) Hilbert scheme (of the relative symmetric products)
is a countable union of projective schemes over $T$ one proves that
the set of points $x\in\kc_t$ with $[x]=c_{\kx_t}$ is a countable union of Zariski closed subsets (see e.g.\ \cite[Ch.\ 22.1]{VoisinHodge}).

Without bounding the order, the result cannot be improved (cf.\ Lemma \ref{lem:inffibre}). For the trivial family
$\kx=X\times T$, Proposition \ref{prop:finite} proves that the locus of $t\in T$ with $\kc_t\subset X$ a constant cycle curve
of order $\leq n$ is Zariski closed. However, the Hodge theoretic line of reasoning seems not to extend to
non-trivial families $\kx\to T$. Only the locus of $t\in T$ with $n\cdot{\rm AJ}_{\rm tr}(\kappa_{\kc_t})=0$ describes a Zariski
closed subset.

%%%%%%%%%%%%%%%

\section{First examples of constant cycle curves}\label{sec:exas}

It is notoriously difficult to construct rational curves on K3 surfaces, see Section \ref{sec:BMMM}. As constant cycle
curves are natural generalizations of rational curves, in particular with respect to (\ref{eqn:BVmostrelevant}),
it seems worthwhile to work out examples. On the one hand, it will become clear that
constant cycle curves are much easier to construct than rational curves, at least if we do not
care about their order. On the other hand, the powerful technique developed in \cite{BHT,LL} to prove existence
of rational curves by reducing modulo $p$ and then using Tate's conjecture  does not seem
to apply to constant cycle curves. The main reason being that a curve could be a constant cycle
curve modulo infinitely many primes without being a constant cycle curve itself. In fact, 
all curves over $\bar\FF_p$  should be constant cycle curves (see Section \ref{sec:finitefields}), 
but not over $\bar\QQ$. 

If not stated otherwise, $X$ will be a K3 surface over an arbitrary algebraically closed field $k$.

\subsection{}
As mentioned before, all rational curves $C\subset X$ are pointwise constant cycle curves. In fact,
we have
\begin{lem}\label{sec:ratcurves}
A rational curve $C\subset X$ in a K3 surface is a constant cycle curve of order one.
\end{lem} 
\begin{proof}
This is easy to see by using (\ref{eqn:loc}), which together with $\Ch^2(X\times\PP^1)\cong
\Ch^2(X)\oplus\Ch^1(X)\cdot h$, where $h$ is the hyperplane section on $\PP^1$, shows $\Ch^2(X)\cong\Ch^2(X\times U)$ for every proper non-empty open subset $U\subset\PP^1$.
Hence $\Ch^2(X\times{k(\eta_C)})\cong\Ch^2(X)$, which is torsion free.
\end{proof}

%%%%%%%%%%%%%%%
\subsection{}\label{sec:BVexas} We shall discuss a construction inspired by 
\cite{VoisinHOG}. In particular, as remarked in \cite[Lem.\ 2.3]{VoisinHOG},
it shows that the union of all constant cycle curves is dense (and even in the classical topology for complex K3 surfaces).

Consider  an elliptic K3 surface with a zero-section
$$\pi:X\to\PP^1,~C_0\subset X.$$ We first  give an ad hoc and geometric description of the constant cycle curves of torsion points on the fibres.
For this assume that $k$ is uncountable. The formal definition for arbitrary algebraically closed field $k$, which is also needed to 
determine the order of $C_n$, is given below. 
Let $$C_n\subset X$$ be the closure of the set of $n$-torsion points on the smooth fibres $X_t$, $t\in\PP^1$.
Let $x_t\in X_t$ denote the origin, i.e.\ $\{x_t\}=C_0\cap X_t$.
Thus, for any $x\in C_n\cap X_t$ the class $[x]-[x_t]$ is $n$-torsion  in $\Ch^1(X_t)$ and hence trivial in $\Ch^2(X)$. Hence, $C_n$ is a pointwise constant cycle curve and by Proposition
\ref{prop:Voisin} in fact a constant cycle curve.

For arbitrary $k$ (as always algebraically closed) these curves (or rather their irreducible components) are constructed as follows: Let $\mu\in\PP^1$ denote the generic point
and $X_\mu$ the generic fibre of $\pi$. For $x\in X_\mu$ with $k(x)/k(\mu)$ finite, let $C_x\subset X$ be the curve obtained as the closure of the point $x\in X_\mu\subset X$. In particular, the generic point $\eta_{C_x}\in C_x$ is just $x\in X_\mu$ and 
 $k(\eta_{C_x})=k(x)$. The diagonal of $C_x$ as a subvariety $\Delta_{C_x}\subset X\times C_x$ is in fact contained
in the surface $X\times_{\PP^1}C_x$. Hence, its class $[\Delta_{C_x}]\in\Ch^2(X\times C_x)$ is the push-forward under
$$\Ch^1(X\times_{\PP^1}C_x)\to\Ch^2(X\times C_x)$$ of the class of the relative diagonal $\Delta_{C_x/\PP^1}\subset X\times_{\PP^1}C_x$. Restricting to the generic point $\eta_{C_x}$ shows that $[\Delta_{C_x}|_{X\times k(\eta_{C_x})}]\in\Ch^2(X\times k(\eta_{C_x}))$
is the push-forward of the class of the point $x\in X_\mu$ in $\Ch^1(X\times_{\PP^1}k(x))=\Ch^1(X_\mu\times_{k(\mu)}k(x))$.

This observation can now be applied to the origin  $x=o\in X_\mu$, i.e.\ the intersection of $C_0$ with $X_\mu$, and
any point $x_n\in X_\mu$ of order $n$. For the latter the associated curve $$C_{x_n}:=\overline{\{x_n\}}\subset X$$ 
 is an irreducible component of the curve $C_n$ of  $n$-torsion points in the fibres as considered above.

\begin{lem}\label{lem:ccctorsion}
i) The curve $C_{x_n}$ is a constant cycle curve of order $d| n$.

ii) The union of all $C_{x_n}$ is dense and, in case $k=\CC$, even dense in the classical topology.
\end{lem}

\begin{proof} The density statement ii) is obvious, as the set of geometric torsion points is dense in the generic fibre $X_\mu$.

For i), note first that $[\Delta_{C_0}]\in\Ch^2(X\times C_0)=\Ch^2(X\times\PP^1)$
is the class of $\{x_0\}\times C_0+C_0\times \{x_0\}$ (for an arbitrary point $x_0\in C_0$), which restricted to $X\times k({o})$
is $[x_0]\times k(o)$. Thus, the restriction of the class of $\Delta_{C_{x_n}}-\{x_0\}\times C_{x_n}$ to $X\times k(\eta_{C_{x_n}})$, which
by definition is $\kappa_{C_{x_n}}$, is the image of $[x_n-o]\in\Ch^1(X_\mu\times_{k(\mu)}k(x_n))$. Here, since
$k(o)=k(\mu)$, the class $[o]\in\Ch^1(X_\mu)$ can be base changed to a class in $\Ch^1(X_\mu\times_{k(\mu)}k(x_n))$.

As $x_n\in X_\mu$ is an $n$-torsion point, the class $[x_n-o]$ is $n$-torsion and hence also its image $\kappa_{C_{x_n}}$.
\end{proof}

\begin{remark} Note that it could happen that $C_{x_n}$ is of order $d<n$, e.g.\ when $x_n$ is a $k(\mu)$-rational point
and, therefore, $C_{x_n}$ a rational curve. But presumably in the generic situation $\kappa_{C_{x_n}}$ is of order exactly $n$.
%\footnote{TBC: In the generic case $C_n$ should be irreducible and of order $n$. Check again by monodromy argument.}
 In any case, this elementary construction already provides
many examples of constant cycle curves which are not rational.
\end{remark}

The construction will now be generalized to covering families of elliptic curves, still following \cite{BV,VoisinHOG}:
Firstly, the existence of nodal rational curves on the generic K3 surface leads to the
existence of a dominating family of elliptic curves for every K3 surface. More precisely, for
an arbitrary K3 surface $X$ there exists a smooth elliptic surface $\kc\to T$ with a surjective morphism $p:\kc\to X$.
See  \cite[Thm.\ 4.1]{HassTsch} for details in the case of characteristic zero.
In positive characteristic use a lift to characteristic zero and reduce the family of curves
back to $p$ which yields a family of elliptic curves if $X$ is not unirational. 
The unirational case being trivial, we shall just ignore it.  Note that then, 
one could
moreover assume that the generic fibre $\kc_t$ is mapped birationally onto its image in $X$, but we will not need this. 

Now, choose an ample rational curve $C_0\subset X$ and  let 
$\widetilde C_0$ be its preimage in $\kc$. Replacing $\kc\to T$ by its base change to  $\widetilde C_0$,
we can assume that $\kc\to T$ admits a section $\sigma_0:T\to\kc$ such that $\sigma_0(T)$ maps
onto a rational curve in $X$, namely $C_0$. Then
consider the curve $C_n\subset X$  defined as the closure of the set of  images of all points
$x\in\kc_t$ in the smooth fibres $\kc_t$, $t\in T$, such that
$n\cdot([x]-[x_{t}])=0$ in $\Ch^1(\kc_t)$, where  $x_{t}$
is the point of intersection of $\kc_t$ with  $\sigma_0(T)$. 
For the same reason as before, the curve $C_n$ is a (possibly reducible) pointwise constant cycle curve.
Note that indeed any torsion point in a smooth fibre $\kc_t$ is contained in a multi-section of
$\kc\to T$ consisting of torsion points in the smooth fibres. 
Also, the arguments in the proof of Lemma \ref{lem:ccctorsion} still apply. 

The following was explicitly stated already in \cite{VoisinHOG} and was also used in \cite{ML}.
It should be compared to Conjecture \ref{conj:ccc}.
\begin{cor} On any K3 surface $X$ over an arbitrary algebraically closed field $k$ the union $$\bigcup_{C =\rm ccc} C\subset X$$
of constant cycle curves (of unbounded order) is dense. For $k=\CC$ density holds in the classical topology.\qed
\end{cor}

%%%%%%%%%%%%%%%%%%%%%%%%%%%
\subsection{}\label{sec:infinell}
In Section \ref{sec:BVexas} we explained how to produce constant cycle curves as multi-sections of elliptic fibrations.
But, as was pointed out by Claire Voisin, also smooth (and hence non-rational)
fibres  can be constant cycle curves.

Consider the Kummer surface $X$ obtained as the minimal resolution of the standard involution on the
product $E_1\times E_2$ of two elliptic curves. Let $$X\to {(E_1\times E_2)/_\pm}\to \PP^1\cong
E_1/_\pm$$ be the elliptic fibration induced by the first 
projection. There are only four singular fibres all of type $I_0^*$ over the two-torsion points. Being
rational, they are constant cycle curves of order one.

Now consider a torsion point $t\in E_1$ of order $n\ne2$ and the fibre $C_t\subset X$ over $\bar t\in\PP^1$. 

\begin{lem}\label{lem:inffibre}
For a torsion point $t\in E_1$ of order $n\ne2$ the associated fibre $C_t\subset X$ is a smooth
constant cycle curve of order $d|n$.
\end{lem}

\begin{proof} Indeed, if $C_t\subset X$ is viewed with respect to the other elliptic fibration $X\to \PP^1\cong E_2/_\pm$ then
it intersects the (smooth) fibres in $n$-torsion points. In particular, it intersects the generic fibre in an $n$-torsion
point $x_n\in E_1\times K(E_2/_\pm)$. Thus, $C_t$ is one of the curves $C_{x_n}$ considered in Lemma
\ref{lem:ccctorsion} and hence a constant cycle curve of order $d|n$.
(It seems likely that the order equals $n$, but for proving this one would need to control
the kernel of $\Ch^1(X_\mu)\to\Ch^2(X\times_k k(\mu))$, see Section \ref{sec:BVexas}.)
\end{proof}

In particular, this construction yields an elliptic K3 surface $X\to\PP^1$ with  infinitely
many fibres that are constant cycle curves (of growing order).
It is not clear to me whether this holds for arbitrary elliptic K3 surfaces.
Also note that a smooth fibre of an elliptic K3 surface $X\to \PP^1$
can even  be a constant cycle curve of order one, see Example \ref{exa:Dillies}
and the discussion in Section \ref{sec:counting}. The above construction dispels hope
expressed in \cite[Sec.\ 17]{Kerr} that the Abel--Jacobi class of a smooth fibre should
in particular always be non-torsion.

%%%%%%%%%%%%%%%%%%%%%%%%%%%
\section{More examples: fixed curves}\label{sec:auto}
We start this section with the branching curve of a double plane which turns out to be
a constant cycle curve of order at most two. This then generalizes to fixed point curves
of arbitrary non-symplectic automorphisms. For simplicity we work in characteristic zero.

\subsection{} Consider a generic double plane, i.e.\
K3 surface $X$  given as a $2:1$ cover $$X\to\PP^2, ~~i:X\congpf X$$ ramified
over a smooth curve $C\subset\PP^2$ of degree six
with $i$ the covering involution. Using the eigenspace decomposition one finds that
$i^*=-{\rm id}$  on $\Ch^2(X)_0$, for $\Ch^2(\PP^2)_0=0$.

Now consider $C$ as a curve in $X$
and write the class $[x]$ of a point $x\in C$ as $[x]=c_X+\alpha_x$ with $\alpha_x\in\Ch^2(X)_0$.
On the one hand, $i^*[x]=[i(x)]=[x]$ and, on the other, $i^*[x]=i^*(c_X+\alpha_x)=c_X-\alpha_x$.
Hence, for $x\in C$, one has $2\cdot\alpha_x=0$ and, since $\Ch^2(X)$ is torsion free, also $\alpha_x=0$,
i.e.\ $C$ is a (pointwise) constant cycle curve. This provides an explicit example of 
 a constant cycle curve which is smooth  and of  genus ten (and so in particular not rational). As we shall see in broader generality,
$C$ is a constant cycle curve of order one or two. This shows already that in general the genus of a constant cycle curve
is not determined by its order. Finding an example of a constant cycle curve of order one that is non-rational
is harder, see Corollary \ref{cor:exacccone}.

%%%%%%%%%%%
\subsection{}
This naive example is now generalized as follows. Suppose $f:X\congpf X$ is an automorphism
of finite order $n$. Assume that  the quotient $$\pi:X\to \bar X:=X/\langle f\rangle,$$ which is possibly singular,
satisfies $\Ch^2(\bar X)_0=0$. Note that due to Bloch's conjecture (cf.\ Section \ref{sec:BlochConj}),
which is known for surfaces of Kodaira dimension $<2$, the latter condition is equivalent to $f^*\ne{\rm id}$ on $H^{2,0}(X)$. Suppose a curve
$C\subset X$ is contained in the fixed point locus ${\rm Fix}(f)$. Then a similar trick as above shows
that $C$ is a (pointwise)
constant cycle curve. Indeed, write $[x]=c_X+\alpha_x$ for $x\in C$. Then $n\cdot [x]=n\cdot c_X+n\cdot\alpha_x$, but on the other hand
$n\cdot[x]$ is the pull-back of $[\pi(x)]\in\Ch^2(\bar X)$. Hence, $n\cdot[x]=n\cdot c_X$, which yields $n\cdot\alpha_x=0$ and, therefore, $\alpha_x=0$.
The calculation in this case suggests that any curve in the fixed point locus of a non-symplectic automorphism of finite order
$n$ is a constant cycle curve of order $d|n$. This can be shown rigorously as follows.

\begin{prop}\label{prop:fixedcurve}
Let $f:X\congpf X$ be an automorphism of finite order $n$  of a K3 surface
$X$ over an algebraically closed field $k$ of ${\rm char}(k)=0$ such that $f^*\ne{\rm id}$ on $H^{0}(X,\Omega_X^2)$. Then any curve $C\subset X$ contained in ${\rm Fix}(f)$ is a constant cycle curve of order $d|n$.
\end{prop}

\begin{proof} Consider the pull-back $\pi^*:\Ch^2(\bar X)\to\Ch^2(X)$
induced by the projection $\pi:X\to \bar X$. Let $y:=\pi(x)$ for a point $x\in X$.
Then $\pi^*[y]=n\cdot c_X$ if $[x]=c_X$ and
$\pi^*[y]=n\cdot[x]$ for any fixed point $x\in X$. The same holds after base change to
any field extension $K/k$. Apply this
to $K=k(\eta_C)$ for a curve $C\subset X$ in the fixed point locus of $f$. Then the generic point $\eta_C\in C$, viewed
as a closed point $\eta_C\in X\times k(\eta_C)$, is fixed under $f_{k(\eta_C)}$. Similarly, the generic point
$\eta_{\bar C}\in\bar C:=\pi(C)\subset \bar X$ can be viewed as a closed point in $\bar X\times k(\eta_{\bar C})$
and, moreover, $k(\eta_{\bar C})= k(\eta_C)$. Then  $[\eta_{\bar C}]\mapsto n\cdot[\eta_C]$
under \begin{equation}\label{eqn:pullquot}
\pi_{k(\eta_C)}^*:\Ch^2(\bar X\times k(\eta_{\bar C}))\to \Ch^2(X\times k(\eta_C)).
\end{equation}

Since $f^*\ne{\rm id}$ on $H^{0}(X,\Omega_X^2)$ and ${\rm kod}(\bar X)<2$, Bloch's conjecture holds true for $\bar X$, i.e.\ $\Ch^2(\bar X)\cong\ZZ$.
In particular, there is a distinguished generator $c_{\bar X}\in \Ch^2(\bar X)$ with $\pi^*(c_{\bar X})=n\cdot c_X$.
In order to conclude, it is therefore enough to prove that $[\eta_{\bar C}]\in \Ch^2(\bar X\times k(\eta_{\bar C}))$
is in the image of the base change map $\Ch^2(\bar X)\to \Ch^2(\bar X\times k(\eta_{\bar C}))$.
Indeed, then the image $n\cdot \kappa_C$ of $[\eta_{\bar C}]-c_{\bar X}\times k(\eta_{\bar C})=0$ 
under (\ref{eqn:pullquot}) would  also be zero.

In other words, it is enough to prove that any curve in $\bar X$ (and so in particular $\bar C$) is a constant cycle curve of order one.
As by assumption $f$ has a fixed curve and hence $\bar X$ is rational, this follows from Proposition \ref{prop:ccconrat}, ii).
\end{proof}

%%%%%%%%%%%%%%%%%%
\subsection{}
Non-symplectic automorphisms have been studied intensively in the literature. If the order is prime,
only $p=2,3,5,7,11,13,17$, and $19$ can occur. Their fixed point loci can be described, which often contains apart from isolated fixed points and rational curves also smooth elliptic curves and even smooth curves of higher genus, see e.g.\ \cite{AST}. However, the genus of constant cycle curves obtained in this way
is rather small, e.g.\ for $p=7,11$ at most elliptic curves can occur and for $p=13,17,19$ all curves
in ${\rm Fix}(f)$ are in fact rational. The maximal genus $g=11$ can be achieved for $p=2$.

\begin{cor}\label{cor:exacccone}
Consider an automorphism $f\in{\rm Aut}(X)$ of order $p\cdot q$ for two primes $p\ne q$. Assume
that $f^p$ and $f^q$ are both non-symplectic, i.e.\ $f^*$ acts by a primitive $(p\cdot q)$-th root of unity on $H^{0}(X,\Omega_X^2)$. Then any
curve $C\subset X$ in  ${\rm Fix}(f)$ is a constant cycle curve of order one.
\end{cor}

\begin{proof} Under the assumptions, the curve $C$ would be in the fixed point locus of both, $f^p$ and $f^q$. Hence, 
$C$ is a constant cycle curve 
of order dividing $q$ and $p$ and, therefore, of order one.
\end{proof}

\begin{exa}\label{exa:Dillies}
As observed by Dillies in \cite{Dillies} and extended by Garbagnati and Sarti
in \cite{GS}, non-symplectic automorphisms of prime order $p$ often occur
as the square $f^2$ of an automorphism $f$ with $f^p$ a non-symplectic involution. The corollary applies
in this situation to curves in ${\rm Fix}(f)$. Cases where a non-rational curve is contained in ${\rm Fix}(f)$ can be found in \cite{Dillies,GS}.

It is worth pointing out that in the explicit example described in  \cite[Sect.\ 7]{Dillies} 
the K3 surface $X$ comes with an elliptic fibration
$\pi:X\to\PP^1$ which is preserved by $f$ and such that one of the smooth(!) fibres is contained
in ${\rm Fix}(f)$.  Concretely, $X$ is the elliptic surface given by the Weierstrass
equation $y^2=x^3+(t^6-1)^2$ and $f(x,y,t)=(x,y,\xi\cdot t)$ with $\xi$ a  primitive sixth root of unity.

As Alessandra Sarti informs me, this example can be generalized to yield a whole
family of elliptic K3 surfaces with a purely non-symplectic automorphism of order six with a smooth elliptic
curve in the fixed point locus and thus being a constant cycle curve of order one.

 For another family of examples see $X_{3,1}$ described in \cite[Prop.\ 4.7]{AS}. It is not difficult to write
down a square root of the order three automorphism given there, but the elliptic structure in the example is not obvious.
\end{exa}

 This then eventually yields an example of a constant cycle curve of order one which is not rational (but smooth elliptic).

\begin{cor}\label{cor:cccorderone}
There exist non-rational constant cycle curves of order one.\qed
\end{cor}

Although it might be difficult to exhibit explicitly constant cycle curves of order one and arbitrary high genus,
there does not seem to be any reason why this should not be possible.

%%%%%%%%%%%%%%%%%%%%%%%%%%%%%%%%%%%%%%%%%%%%%%%
\section{Bitangent correspondence}\label{sec:bitangentcorr} Here, we exhibit a more involved
example, close in spirit to the ones described in Section \ref{sec:auto},
which leads to constant cycle curves of order at most four and geometric genus $201$.
These curves will be constructed as the fixed locus of the `bitangent correspondence' for a generic quartic K3 surface. The bitangent correspondence maps a generic point to the second
contact point of a bitangent at $x$. Since generically there are six bitangents at every point, this does not
define a map, but we will show that it is well defined on the Chow ring. Its fixed locus is the curve of contact points of hyperflexes.

\smallskip

Recall that for a quartic $X\subset \PP^3$ a line $\ell\subset\PP^3$ is called a \emph{bitangent} of $X$ if
at every $x\in X\cap \ell$ the intersection multiplicity is at least two. A bitangent $\ell$ is a \emph{hyperflex} if
there is a unique point of intersection (and, clearly, the intersection multiplicity is four then).

In this section we shall work over $\CC$.

\subsection{} Consider a smooth quartic $X\subset\PP^3$ not containing a line.  For generic $x\in X$ the 
curve $$C_x:=T_xX\cap X$$ has exactly one singularity, a node at $x$.
Let $\nu:\widetilde C_x\to C_x$ be its normalization and $f:\widetilde C_x\to X$ its composition with the inclusion
$C_x\subset X$. Then $\widetilde C_x$ is a smooth curve 
of genus two.

Choose a generic line $\PP^1\subset T_xX\cong\PP^2$ and consider the linear projection $\varphi:C_x\to\PP^1$  from the node $x\in C_x$.
As $\deg(C_x)=4$ and $x\in C_x$ is a node, $\varphi$ is of degree two and so is the composition
$\tilde\varphi:\widetilde C_x\to \PP^1$. By Hurwitz formula, the ramification divisor of $\tilde\varphi$ is of degree six. Thus, for generic
choices there are exactly six lines $$\ell_1,\ldots,\ell_6\subset T_xX$$ passing through $x$ and such that they
are bitangent to $C_x$ at some other point
$y_1,\ldots,y_6\in C_x$. (This is classical and well known. All it is saying is that there are exactly six bitangents through a generic $x\in X$.) 

The construction in particular shows that up to two-torsion the points $y_1,\ldots,y_6\in\widetilde C_x$
are linearly equivalent, as  $\tilde\varphi^*\ko(1)\cong\ko(2\cdot y_i)$ for all $i$.
Thus, for $f_*:{\rm Pic}(\widetilde C_x)\to {\rm CH}^2(X)$ one finds
$f_*\ko(1)=2\cdot[y_i]\in\Ch^2(X)$ and hence $2\cdot [y_1]=\ldots=2\cdot[y_6]$. Since $\Ch^2(X)$ is  torsion free,
in fact $[y_1]=\ldots=[y_6]\in\Ch^2(X)$.

So the $1:6$ correspondence $x\mapsto \{y_1,\ldots,y_6\}$ induces a well-defined
involution(!) $$\gamma:\Ch^2(X)\congpf\Ch^2(X),~~[x]\mapsto[y_1].$$
In fact, $6\cdot\gamma=[\Gamma_X]_*$, where  $\Gamma_X$ is the closure of
the locus $\{(x,y)~|~x\ne y,~\overline{x,y} \text{~bitangent}\}\subset X\times X$.

Note that for $(x,y)\in\Gamma_X$, i.e.\ generically $\overline{x,y}$ is a bitangent, the class
$\alpha:=[x]-[y]$ satisfies $\gamma(\alpha)=-\alpha$. This can be proved in general.

\begin{prop}\label{prop:bitminus}
The bitangent correspondence $\gamma$ acts by $-{\rm id}$ on $\Ch^2(X)_0$.
\end{prop}

\begin{proof} Write $[x]=c_X+\alpha_x$ for any $x\in X$. We have to show $\gamma(\alpha_x)=-\alpha_x$.
Since any point is rationally equivalent to a cycle contained in a fixed non-empty open subset, we can 
assume that $x$ is generic as above. For $\tilde\varphi:\widetilde C_x\to\PP^1$, Hurwitz formula yields
$\omega_{\widetilde C_x}\cong\tilde\varphi^*\ko_{\PP^1}(-2)\otimes \ko(\sum y_i)$. On the other hand, 
$\omega_{\widetilde C_x}\cong\nu^*\omega_{C_x}\otimes\ko(-x_1-x_2)$, where $x_1,x_2\in\widetilde C_x$ are the two points over the node $x\in C_x$. Since $\omega_{C_x}\cong\omega_{T_xX}\otimes\ko_{T_xX}(4)\cong\ko_{\PP^3}(1)|_{C_x}$, one obtains
$$\ko\left(\sum y_i\right)\cong\nu^*(\ko_{\PP^3}(1)|_{C_x})\otimes\tilde\varphi^*\ko_{\PP^1}(2)\otimes\ko(-x_1-x_2)\cong \nu^*(\ko_{\PP^3}(3)|_{C_x})\otimes\ko(-3\cdot x_1-3\cdot x_2),$$
where one uses $\tilde\varphi^*\ko_{\PP^1}(1)\cong\nu^*(\ko_{\PP^3}(1)|_{C_x})\otimes\ko(-x_1-x_2)$.

But then (using (\ref{eqn:BV}))
 $$[\Gamma_X]_*[x]=f_*\left(\sum y_i\right)=3\cdot(h.C_x)-6\cdot [x]=6\cdot c_X-6\cdot\alpha_x.$$ For a point $x$ with $\alpha_x=0$ this shows
$[\Gamma_X]_*c_X=6\cdot c_X$ and then for arbitrary $x\in X$ also $[\Gamma_X]_*\alpha_x=-6\cdot\alpha_x$.
\end{proof}

\begin{cor}\label{cor:bitH20}
The bitangent correspondence $\gamma$ acts by $-{\rm id}$ on $H^{2,0}(X)$.
\end{cor}

\begin{proof} This follows from the `easy direction'
of the general conjectures on the Bloch--Beilinson filtration, see e.g.\ \cite[Prop.\ 23.18]{VoisinHodge}.
See also Remark \ref{rem:moreproof}.
\end{proof}

%%%%%%%%%%%%%%%%%%%%%%
\subsection{} There is a more geometric way of defining this correspondence. Consider
the universal family of bitangents:
$$\xymatrix{B_X\ar[d]_q\ar[r]^p&X\\
F_X.&}$$
More explicitly, bitangents $\ell$ of $X$ correspond bijectively to points $[\ell]\in F_X$ and $B_X$ is the variety
of all $(\ell,x)$, with $\ell$  a bitangent of $X$ and $x$ is a point of contact, i.e.\ $x\in X\cap\ell$.
By $p$ and $q$ we denote the projections $(\ell,x)\mapsto x$ resp.\ $(\ell,x)\mapsto [\ell]$. In particular, $q$
is of degree two with a well studied ramification divisor $D_{\rm hf}\subset F_X$  (the curve of hyperflexes)
and $p$ is of degree six, as there are exactly six bitangents at a generic $x$.
As shown by Tikhomirov and Welters, $F_X$ and $B_X$ are smooth irreducible surfaces (of general type), see
\cite{Tikho, We}.

Now consider the covering involution $$i:B_X\to B_X$$ of
the double cover $q:B_X\to F_X$. So for generic bitangent $\ell$, the involution $i$ interchanges
$(\ell,x)$ and $(\ell,y)$, where $X\cap \ell=\{x,y\}$.

Then $\Gamma_X\subset X\times X$ is the image of the natural morphism
$$g:B_X\to X\times X,~(\ell,x)\mapsto (x=p(\ell,x),p(i(\ell,x))).$$
In other words, $g(\ell,x)=(x,y)$ for $\ell=\overline{x,y}$ with $x,y\in X$.

The morphism $g$ is generically injective, as the points of contact $x,y$ of a bitangent $\ell$ clearly
determine $\ell$ when $x\ne y$.  

\begin{lem}\label{lem:eqncorr}
The correspondence $[\Gamma_X]_*=6\cdot\gamma:{\rm CH}^2(X)\to {\rm CH}^2(X)$ coincides with
$$p_*\circ i^*\circ p^*:{\rm CH}^2(X)\to{\rm CH}^2(B_X)\to\Ch^2(B_X)\to{\rm CH}^2(X).$$ The same
assertion of course holds on the level of cohomology.\qed
\end{lem}

Mapping a bitangent $\ell\in F_X$ to its intersection $X\cap \ell$ defines a closed embedding
$$F_X\,\hookrightarrow {\rm Hilb}^2(X).$$
The image is the fixed point locus of the Beauville involution mapping $[Z]\in {\rm Hilb}^2(X)$ to the
residual intersection of the unique line through $Z$ with $X$.

\begin{cor} 
The surface $F_X\subset{\rm Hilb}^2(X)$ is a rigid Lagrangian surface of general type.
\end{cor}

\begin{proof} That $F_X$ is Lagrangian is an immediate consequence of $\gamma$ acting as $-{\rm id}$ on $H^{2,0}(X)$.
It is rigid, because $H^0(F_X,{\mathcal N}_{F_X/{\rm Hilb}^2(X)})\cong H^0(F_X,\Omega_X)=0$ due to \cite[(3.43)]{We}.
Its canonical bundle $\omega_{F_X}$ is actually ample, as was shown in \cite{Tikho,We}.
\end{proof}

\begin{remark}\label{rem:moreproof}
Using the geometric description of the bitangent correspondence, one can give 
 another, more roundabout, proof of Corollary \ref{cor:bitH20}.
Suppose $\gamma^*\ne-{\rm id}$ on $H^{2,0}(X)$. Then $q_*\circ p^*: H^{2,0}(X)\to H^{2,0}(F_X)$ does not vanish.
By the work of Tikhomirov  \cite{Tikho} and Welters \cite{We}, there exists an \'etale cover $\pi:F\to F_X$ of degree two
with the property that $\bigwedge^2H^1(F,\QQ)\cong\pi^*H^2(F_X,\QQ)$. Here, $F$ is the Fano variety of lines
on the double quartic solid $Y\to\PP^3$ branched over $X\subset\PP^3$. This would establish a non-trivial algebraic(!)
correspondence between $H^2(X,\QQ)$ and $H^1(F,\QQ)\otimes H^1(F,\QQ)$, i.e.\ between the K3 surface $X$ 
and the square of
the abelian variety given by the weight-one Hodge structure of $F$. This is reminiscent of the Kuga--Satake corres\-pon\-dence
which is conjectured to be algebraic. However, as Fran\c{c}ois Charles explained to me, the ten-dimensional abelian variety determined by $H^1(F,\QQ)$ is too small to play a part in the Kuga--Satake correspondence for the generic
quartic (which would be of the form $A^{2^{10}}$ with $A$ a simple abelian variety of dimension $2^9$), so that the uniqueness of the Kuga--Satake correspondence eventually leads to a contradiction. Hence, $\gamma^*=-{\rm id}$
on $H^{2,0}(X)$.
\end{remark}

\begin{remark}
Once Corollary \ref{cor:bitH20} has been verified, it can in turn be used to prove Proposition \ref{prop:bitminus} by considering the family
of all quartics $\kx\to|\ko(4)|$. Either one uses the explicit description for the Chow ring of $\Ch^*(\kx)$ by viewing
$\kx$ as a projective bundle over $\PP^3$ or  a technique developed in \cite{VoisinCatanese} which only uses that $\kx$ is rationally
connected.
\end{remark}

%%%%%%%%%%%%%%%%%%%%%

\subsection{} Consider now the curve of hyperflexes $D_{\rm hf}\subset F_X$ and the image of the curve $q^{-1}(D_{\rm hf})\subset B_X$ under the projection $p:B_X\to X$.
This yields the curve $$C_{\rm hf}:=p(q^{-1}(D_{\rm hf}))\subset X$$ of all points $x\in X$ such that there exists a hyperflex at $x$. 

\begin{prop}\label{cor:cccbit}
For a quartic $X\subset\PP^3$ not containing a line, the curve $C_{\rm hf}\subset
X$ of contact points of hyperflexes is a constant cycle curve of order $n|4$.
\end{prop}

\begin{proof} We shall first give two pointwise arguments showing that $C_{\rm hf}$ is a constant
cycle curve. The first one seems to (wrongly?) suggest that the order should be at most two, whereas
the second one can be turned into a rigorous argument proving the assertion.

-- Any hyperflex $\ell$  at a point  $x\in C_{\rm hf}$ is the limit of proper bitangents
$$\ell_t:=\overline{x_t,y_t}\to \ell$$ with   $x_t,y_t$ both
specializing to $x$. Since $\gamma([x_t])=[y_t]$,
 Proposition \ref{prop:bitminus} shows $[x_t]+[y_t]=2\cdot c_X$ which after specializing
 gives $2\cdot[x]=2\cdot c_X$. Therefore, $C_{\rm hf}$ is a pointwise constant cycle curve and hence, by Proposition \ref{prop:Voisin}, also a
constant cycle curve. Note that the argument cannot be used to actually prove that the order is $n|4$ (or even better $n|2$),
as only $6\cdot \gamma$ is  a priori defined by an integral cycle and running the argument again with $6\cdot\gamma$ 
only shows that the order divides $12$.

-- For $x\in C_{\rm hf}$ there exists a line $\ell\subset T_xX$ with $\ell$ and  $C_x$ only intersecting in $x$ (with multiplicity four).
But $\ell= T_xX\cap H$ for some hyperplane $H\subset \PP^3$. Hence, $4\cdot [x]=C_x.(H|_X)=4\cdot c_X$ by (\ref{eqn:BV}).

-- The last argument works for the base change $X_k$ to any field $k/\CC$ (not necessarily
algebraically closed), as long as the hyperflex $\ell$ is unique.
Indeed, $T_xX_k$ and $C_x\subset X_k$ are defined over $k$ and if $\ell\subset T_xX_{\bar k}$ is the unique hyperflex, it is left invariant under
$Gal(\bar k/k)$. Hence, $\ell$ is and $H$ can be defined over $k$. Thus, $4\cdot[x]=4\cdot c_{X_k}$ in $\Ch^2(X_k)$. This can be applied
to the generic point $\eta$ of $C_{\rm hf}\subset X$ viewed as a closed point $\eta\in X\times{k(\eta)}$.
Since $q^{-1}(D)\to C_{\rm hf}$ is generically injective (see the proof of Proposition \ref{prop:CHF}), the generic
point of $C_{\rm hf}$ admits a unique hyperflex.
This proves that $\kappa_{C_{\rm hf}}$ is of order $n|4$.
\end{proof}

Continuing the analogy between the covering involutions of $X\to\PP^2$ and $B_X\to F_X$ in the degree two resp.\
four case, the curve $C_{\rm hf}\subset X$ should be seen as the analogue of the degree six branching curve
$C\subset X\to\PP^2$. It is curious to note that
as in the degree two case, the generic K3 surface of degree four contains a distinguished curve. It would be interesting
to find distinguished curves in K3 surfaces of higher degree curve as well.\footnote{To venture a guess, any distinguished curve, i.e.\ a curve that is naturally defined in all generic K3 surfaces of fixed degree, should be a constant cycle curve. Compare this to
O'Grady's Franchetta conjecture in \cite{OG}.}

%%%%%%%%%%%%%%%%%%%%%%%%%%%%%%%%%%%%%%
\subsection{}
A little more can be said about the curve $C_{\rm hf}\subset X$, which shall be recorded here.

\begin{prop}\label{prop:CHF}
For a quartic $X\subset \PP^3$ not containing a line the curve
$C_{\rm hf}\subset X$ is a singular irreducible curve in the linear system
$|\ko_X(20)|$ of geometric genus $201$.  
In particular, $C_{\rm hf}$ is not rational.
\end{prop}

\begin{proof} By construction, $C_{\rm hf}$ is the image of $q^{-1}(D_{\rm hf})\cong D_{\rm hf}\subset F_X$, 
where $q:B_X\to F_X$ is a $2:1$ cover with ramification curve $D_{\rm hf}\subset F_X$.
By \cite[(1.6)\&\! p.\!\! 40]{We} the curve of hyperflexes $D_{\rm hf}\subset F_X$ is a smooth curve of genus $201$.

A dimension count similar to the one in \cite[p.\ 17-18]{We} shows that $q^{-1}(D_{\rm hf})\to C_{\rm hf}$ is ge\-neri\-cally injective.
Also, using the notation of \cite[Sec.\ 3]{We}, one computes the intersection number
$p^*h.[q^{-1}(D_{\rm hf})]=p^*h.(\sigma+p^*h)=80$, where $h$ is the hyperplane section of $X$ and $\sigma$ the hyperplane
section on the space of lines. Since $h.h=4$, this shows
$[C_{\rm hf}]=20\cdot h$, i.e.\ $C_{\rm hf}\in|\ko_X(20)|$.
\end{proof}

\begin{remark} Coming back to the variety of bitangents $F_X\subset{\rm Hilb}^2(X)$.
As pointed out by Kieran O'Grady, ${\rm Hilb}^2(X)$ can be seen as a degeneration of an EPW-sextic,
such that $F_X$  corresponds to the singular locus of the sextic. In the same spirit,
whenever an EPW-sextic is of the form
${\rm Hilb}^2(X')$ for some K3 surface $X'$, then the singular locus of the sextic is likely to give
an interesting antisymplectic auto-correspondence for $X'$. This can be applied to a generic K3 surface of genus
six, see \cite[Sec.\ 4]{OG2}. This should eventually lead to a picture quite similar to the one described here. In particular,
one finds as an analogue of the curve $C_{\rm hf}$ the curve of points in $X'$ with a conic in a certain ambient
Fano threefold with higher contact at this point which is expected to be a constant cycle curve.
\end{remark}

\begin{remark} We conclude by making two remarks on the dynamical aspects of the bitangent correspondence.

i) On may wonder whether the bitangent correspondence $x\mapsto y_i$ can play the role of an endomorphism
in other respects as well. E.g.\ endomorphisms of K3 surfaces have been  used to  prove potential density
of rational points on K3 surfaces
over number fields. Although the universal family of bitangents is also defined over the same field
as $X$, the construction seems of little use for this purpose. Indeed the surface $B_X$ is of general type and Lang's
conjecture would predict much fewer rational points on $B_X$ than on $X$ itself. In other words, mapping $x\in X$
to the other points of contact $y_i\in X$ of  bitangents at $x$, increases the residue field.

ii) If $C\subset X$ is a constant cycle curve in an arbitrary K3 surface and $f:X\congpf X$
is any automorphism, then $f(C)\subset X$ is again a constant cycle curve (of the same order).
But, since the generic K3 surface does not admit non-trivial automorphisms, this method
fails in the generic case. However, for the generic quartic $X\subset \PP^3$ the bitangent correspondence
can sometimes replace the missing automorphisms. Indeed, if $C\subset X$ is a constant cycle curve,
then  $p(i(p^{-1}(C)))$ is  again a constant
cycle curve. This can be applied to e.g.\ $C=C_{\rm hf}$, in which case $p^{-1}(C)=q^{-1}(D_{\rm hf})\cup C'$. The
curve $q^{-1}(D_{\rm hf})$ is of course invariant under the involution $i$, but $C'$ is not and produces
a new constant cycle curve $p(i(C'))$ in $X$. 
\end{remark}

\section{Finite fields}\label{sec:finitefields}
In this section we study K3 surfaces over finite fields $\FF_q$, $q=p^d$, with $p\ne2$.
Contrary to the case of K3 surfaces over $\bar\QQ$, it is easy to see that $\Ch^2(X)\cong\ZZ$
for any $X$ over $\bar\FF_p$. Indeed, any finite collection of points is contained
in some curve $C\subset X$. As one always finds a finite field $\FF_q\subset\bar\FF_p$
over which the finitely many points and the curve are defined, the finiteness of $\FF_q$-rational points of
$\Pic^0(C)$ and the torsion freeness of $\Ch^2(X)$ (over the algebraically closed field $\bar\FF_p$) are enough
to conclude.

But the Bloch--Beilinson philosophy in positive characteristic  also predicts  $\Ch^2(X)\cong\ZZ$
for $X$ over the algebraic closure of
$\FF_p(t)$. Surprisingly, explicit examples of K3 surfaces over function fields verifying the conjecture
can actually be worked out, unlike the case of K3 surfaces over number fields.

\begin{remark} For $k=\bar\FF_p$ the Chow group is not expected to increase when passing from $X$ over $k$
to $X\times_k{\overline{k(t)}}$. This does not hold for other fields. Indeed, similarly to Bloch's argument showing
that $\Ch^2(X)\to\Ch^2(X\times_k{k(\eta_X)})$ is not surjective for a K3 surface $X$ in arbitrary characteristic as long as 
 $\rho(X)\ne22$, Green, Griffiths, and Paranjape showed in \cite{GGP} that in  characteristic zero the Chow group
grows already when passing to an algebraically closed field of transcendence degree one over the base field. Cf.\ Corollary
\ref{cor:BGGPweak} for a weaker version that also works in positive characteristic.
\end{remark}

\subsection{}
Due to $\Ch^2(X)\cong\ZZ$ for a K3 surface $X$ over $\bar\FF_p$, any curve $C\subset X$ is a pointwise constant
cycle curve. But the stronger statement is also conjectured to hold due to the following

\begin{prop}\label{prop:BBCff}
Let $X$ be a K3 surface over $\bar\FF_p$. Then $X\times_{\overline{\FF}_p}{\overline{\FF_p(t)}}$
satisfies the Bloch--Beilinson conjecture, i.e.\ $\Ch^2(X\times_{\overline{\FF}_p}{\overline{\FF_p(t)}})\cong\ZZ$, if and only if
every curve $C\subset X$ is a constant cycle curve.
\end{prop}

\begin{proof} Clearly, $\Ch^2(X\times_{\overline{\FF}_p}{\overline{\FF_p(t)}})\cong\ZZ$ if and only
if the pull-back
\begin{equation}\label{eqn:pb}
\Ch^2(X)\to\Ch^2(X\times_{\overline{\FF}_p}{\overline{\FF_p(t)}})
\end{equation}
 is an isomorphism. Due to Lemma
\ref{lem:BVclass2}, a curve $C\subset X$ is a constant cycle curve if and only if
$[\eta_C]\in\Ch^2(X\times{\overline{k(\eta_C)}})$ is in the image of the pull-back. Assuming the Bloch--Beilinson
conjecture, the latter now follows from choosing an embedding $\bar\FF_p\subset k(\eta_C)\subset \overline{k(\eta_C)}\cong\overline{\FF_p(t)}$.

Conversely, a closed point $x\in X\times{\overline{\FF_p(t)}}$ either projects to a closed point in $X$
or the generic point of a curve $C\subset X$. In the first case, $[x]$ is in the image of (\ref{eqn:pb}),
whereas in the latter $[x]=[\eta_C]$ with the generic point $\eta_C\in C$ viewed as a closed point
of $X\times{\overline{k(\eta_C)}}=X\times{\overline{\FF_p(t)}}$. But if $C$ is a constant cycle curve, then up to torsion
$[\eta_C]$ is in the image of $\Ch^2(X)\to\Ch^2(X\times{k(\eta_C)})$ and, therefore, $[x]$ is in the image
of (\ref{eqn:pb}).
\end{proof}

We stress again, that due to \cite{GGP} the corresponding statement is false for $X$ over $\bar\QQ$, i.e.\ there
always exist curves $C\subset X$ which are not constant cycle curves. However, according to Conjecture \ref{conj:BBC}
all curves are  expected
to be pointwise constant cycle curves and it seems likely that every point is at least contained in 
a constant cycle curve (cf.\ Conjecture \ref{conj:BBCgeom}). Note that any linear system 
$|L|$ could contain countably many constant cycle curves. So a priori the finiteness of constant cycle curves of
bounded order does not prove the existence of curves that are not constant cycle curves. The existence
 for surfaces over $\bar\QQ$ as proved in \cite{GGP} eventually relies on results of Terasoma and in fact shows
 the existence of infinitely many such curves in any ample linear system.

 \subsection{} The Bloch--Beilinson conjecture for function fields is related to the conjectured
 finite-dimen\-sionality in the sense of Kimura and O'Sullivan.  To be more precise, consider a
 K3 surface $X$  and a curve $C\subset X$ over a finite field $\FF_q$. Suppose Kimura--O'Sullivan finite-dimensionality
 holds for $X$ and $\widetilde C$ and hence for $X\times \widetilde C$ (cf.\ \cite[Prop.\ 5.10]{Kimura}).
 Then the Tate conjecture $T^1$ for $X$ (using the recent  \cite{Charles,Keerthi,Maulik}) and $\widetilde C$, which implies
 the Tate conjecture $T^1$ for $X\times \widetilde C$ and by duality the Tate conjecture $T^2$ for $X\times \widetilde
 C$ (see \cite{Tate} or
 \cite[Cor.\ 2.2, Thm.\ 1.4]{Milne}),
 yields $$\Ch^2(X\times\widetilde C)\otimes_\ZZ\QQ_\ell\congpf H^4((X\times\widetilde C)_{\bar\FF_q},\QQ_\ell(2))^{Gal(\bar\FF_q/\FF_q)},$$ see \cite[Thm.\ 0.4,
 Thm.\ 5.2]{Jannsen}, \cite[Thm.\ 1]{Kahn}, or \cite[Thm.\ 4.2]{Andre}. The cycle $Z_C$ is homologically trivial, i.e. $$0=[Z_C=\Delta_{f_C}-C\times[L_0]-\{x_0\}\times \widetilde C]\in H^4((X\times\widetilde C)_{\bar\FF_q},\QQ_\ell(2))^{Gal(\bar\FF_q/\FF_q)}$$ (see Section \ref{sec:ZC}) and, under the assumption of Kimura--O'Sullivan finite-dimensionality, the cycle $Z_C$ on $X\times \widetilde
 C$ must then also be rationally equivalent to zero up to torsion. Hence, its restriction $\kappa_C\in\Ch^2(X\times{k(\eta_C)})$ is also torsion. Thus, the Bloch--Beilinson conjecture holds for the K3 surface $X\times{\FF_q(t)}$, i.e.\ $\Ch^2(X\times{\overline{\FF_q(t)}})\cong\ZZ$.  
 
 \smallskip
 
 This remark immediately produces actual examples of K3 surfaces over function fields for which the
 Bloch--Beilinson conjecture can be confirmed. Note that this is in contrast to the number field case where
 not a single K3 surface is known to satisfy $\Ch^2(X_{\bar\QQ})\cong\ZZ$. It would also be very interesting to find
 an example of a non-isotrivial K3 surface over $\overline{\FF_p(t)}$ satisfying the Bloch--Beilinson conjecture.

 \begin{exa}
Kimura--O'Sullivan finite-dimensionality is known for rational Chow motives in the tensor subcategory generated by Chow motives of abelian varieties, cf.\ \cite[Ex.\ 9.1]{Kimura}. Also, quotients of finite-dimensional Chow motives are again finite dimensional.
Hence, Kummer surfaces are known to be finite-dimensional. Therefore, any Kummer surface $X$ over $\FF_q$ yields
a K3 surface $X_{\FF_q(t)}$ over $\FF_q(t)$ which satisfies the Bloch--Beilinson conjecture, i.e.\ $\Ch^2(X\times{\overline{\FF_q(t)}})\cong\ZZ$. 
Finite dimensionality of Kummer surfaces is eventually deduced
from the fact that they are dominated by products of curves and, in fact, in \cite{Schoen} the
Bloch--Beilinson conjecture for $X\times {\FF_q(t)}$ was proved for any variety  $X$ dominated by a product of curves. 

The reduction modulo $p$ of Voisin's example in \cite[Sect.\ 3.3]{VoisinRemarks}
is likely to provide other examples. See \cite{Pedrini} for more on finite-dimensionality of K3 
surfaces (in characteristic zero).
\end{exa}
 
\begin{prop}\label{cor:allccconKummer}
Let $X$ be a K3 surface over a finite field $\FF_q$ for which Kimura--O'Sullivan finite-dimensionality holds,
e.g.\ $X$ a Kummer surface. Then every curve in $X_{\bar\FF_q}$ is a constant
cycle curve.\qed 
\end{prop}

By  Proposition \ref{prop:ccconrat}, the result holds also true for unirational (and hence by \cite{Liedtke} for all supersingular)  K3 surfaces.
 
\begin{remark}
In the situation of the corollary and assuming that $X$ is not unirational, it would be interesting to decide whether  among the curves in a fixed linear system $|L|$, which are all constant cycle curves, there are at most finitely many of bounded order, i.e.\ whether the analogue of Proposition \ref{prop:finite} holds.
Also, is there an explicit example (for which Kimura--O'Sullivan finite-dimensionality is known) where one can show
that all curves are constant cycle curves directly (i.e.\ geometrically)?
 \end{remark}

\subsection{}  We conclude this section with a result lending evidence to  Conjecture \ref{conj:BBCgeom}.

 \begin{prop}\label{prop:cccallover}
 Every closed point $x\in X$ in a K3 surface over $\bar\FF_p$
 is contained in a constant cycle curve $x\in C\subset X$.
 \end{prop}
 
\begin{proof}
If $X$ is unirational, then every point $x\in X$ is in contained in a rational curve, which yields the assertion.
Thus, we may assume that $X$ is not unirational.

Next we  use the construction of Section \ref{sec:BVexas}. So we pick a smooth elliptic surface 
$\kc\to T$ with a surjective morphism $p:\kc\to X$.
As before,  by a further base change,
we can assume that  $\kc\to T$
comes with a zero-section that maps to a constant cycle curve in $X$. Now, all singular fibres are rational and,
therefore, map to constant cycle curves in $X$. A point $x\in\kc_t$ in a smooth fibre
is automatically torsion and hence contained in one of the curves
$C_n\subset\kc$ of fibrewise torsion points.
By Section \ref{sec:BVexas}, the curves $C_n$ are constant cycle curves. This 
proves the assertion.
 \end{proof}

Note that a stronger result for Kummer surfaces, namely the existence of a rational(!) curve through every point, has been
proved already by Bogomolov and Tschinkel in \cite{BT2}. The assumption that $X$
is Kummer is in \cite{BT2} used for an explicit geometric construction for Jacobian
Kummer surfaces.

%%%%%%%%%%%%%%%%%%%%%%%%%%%%%%%%%%%

\section{Further questions and remarks}

\subsection{} Let $X$ be a  K3 surface over an arbitrary algebraically closed
field $k$.  It seems not impossible that any closed point $x\in X$ with $[x]=c_X$ is in fact contained in a constant
cycle curve $C\subset X$ (see Conjecture \ref{conj:BBCgeom}). Corollary \ref{cor:cccthroughpt} provides some weak evidence. Note however that this should fail in general (probably whenever $k$ is  uncountable and algebraically closed)
when  constant cycle curves are replaced by rational curves. E.g.\ for a complex K3 surface
and a non-rational constant cycle curve $C\subset X$ (examples have been given above), only countably
many of the points $x\in C$ can be contained in some rational curve. Indeed, there are at most
countably many rational curves contained in $X$.

%%%%%%%%%%%%%%%%%%%%%
\subsection{}\label{sec:counting}

Due to Proposition \ref{prop:finite}, on a complex K3 surface only finitely many curves
in a fixed linear system $|L|$ can be constant cycle curves of a given order $n$. This prompts two questions:

i) Can the number be determined, e.g.\ in geometrically interesting situations?

ii) Is this number deformation
invariant?

\smallskip

Both aspects have been addressed extensively for rational curves in the context of the Yau--Zaslow
conjecture, see \cite{KMPS}. Once the counting is done properly, the number of rational curves (counted
with multiplicities) is deformation invariant and is given by a certain generating series.

For the time being it is not at all clear (to me) whether a similar picture is to be expected 
for constant cycle curves. In fact, the observation that apart from rational curves also non-rational
curves (even smooth ones) can occur as constant cycle curves of order one is a little suspicious. One would
need to add those to the (Gromov--Witten) number of rational curves.

%%%%%%%%%%%%%%%%%%%%%
\subsection{} It is to be expected that an analogous picture will emerge for curves $C\subset X$ that are special
with respect to O'Grady's filtration (see \cite{OG}):
$$S_0(X)\subset S_1(X)\subset \ldots\subset S_d(X)\subset\ldots\subset \Ch^2(X).$$
Here, $S_d(X)$ is the set of cycles that can be written as $[Z]+m\cdot c_X$
with $Z$ an effective cycle of degree $d$. Equivalently, $\alpha\in S_d(X)$ if and only if there exists 
a possibly reducible curve $C\subset X$ such that $g(\widetilde C)\leq d$ and
$\alpha\in{\rm Im}(f_{C*}:\Pic(\widetilde C)\to\Ch^2(X))$ (see \cite{OG,VoisinHOG}). In particular, $S_0(X)=\ZZ \cdot
c_X$.

Note that the kernel of $f_{C*}:\Pic^0(\widetilde C)\to\Ch^2(X)$ is a countable union of translates
$L_i+A_i$ of abelian subvarieties $A_i\subset \Pic^0(\widetilde C)$. One defines $\dim({\rm Ker}(f_{C*}))$
as the minimum of all dimensions $\dim(A_i)$. 
So, $\dim({\rm Ker}(f_{C*}))=g$ if and only if $C$ is a constant cycle curve (we work over $\CC$),
which is equivalent to ${\rm Im}(f_{C*})\subset S_0(X)$. More generally,
it is not difficult to see that $\dim({\rm Ker}(f_{C*}))\geq g-d$ implies ${\rm Im}(f_{C*})\subset S_d(X)$ and
I would expect that with the techniques from \cite{VoisinSympl} also the converse can be proved.

Then in any linear system $|L|$ the locus of curves $C$ with ${\rm Im}(f_{C*})\subset S_d(X)$ 
should be a countable union of Zariski closed subsets, but in order to get honest Zariski closed subsets,
one would first need to introduce the analogue of the order of a constant cycle curve and then restrict to those
of finite order.

For a lack of a better name one could call a curve $C\subset X$, say integral, \emph{$d$-special}
if ${\rm Im}(f_{C*})\subset S_d(X)$. Of course, there would not be anything special about
$C$ for $d\geq g(\widetilde C)$ and $0$-special curves would simply be constant cycle curves.
So the first question one needs to address is how to define the order of a $d$-special curve
for $d<g(\widetilde C)$. 

%%%%%%%%%%%%%%%%%%%%%%%

\subsection{} We have not yet explored constant cycle curves from the infinitesimal point of view.
As explained by Bloch, $H^2(X,\ko_X)\otimes\Omega_{\CC/\QQ}$ should be viewed
as the `tangent space' of $\Ch^2(X)$. For any curve $C\subset X$ the
induced $f_{C*}:\Pic^0(\widetilde C)\to \Ch^2(X)$ induces a natural map between the tangent spaces
$H^1(\widetilde C,\ko_{\widetilde C})\to H^2(X,\ko_X)\otimes\Omega_{\CC/\QQ}$. For a constant cycle curve the map
should be trivial and, as addressed in \cite{GGBook}, it is interesting to see how much information
about $f_{C*}$ is actually encoded by the derivatives (at all points). Is there anything special about the derivative
in points $x\in C$ with $[x]=c_X$?

\section{Appendix  by Claire Voisin}
The aim of this appendix is to prove Conjecture \ref{conj:ccc} for the generic (and not
only general!) K3 surface. More precisely, we prove

\begin{thm}\label{thm:app}
Let $X$ be  a complex K3 surface admitting a non-isotrivial
one-parameter family of elliptic curves which is not an elliptic pencil. Then X
admits infinitely many constant cycle curves  of bounded order, the union of
which is dense in the classical topology.
\end{thm}

Since the assumptions are Zariski open conditions and families of elliptic curves satisfying the assumptions
can be found for the general K3 surface of Picard number one, the theorem does imply Conjecture \ref{conj:ccc}
for the generic complex K3 surface.

\subsection{}\label{sec:notation}
To set up notation, we spell out the assumptions of the theorem: We assume that there exist
a smooth elliptic surface $q:\kc\to T$ and a surjective morphism $p:\kc\to X$ with the following properties:

i) The family $q:\kc\to T$ is not isotrivial, i.e.\ there does not exist any dominant quasi-finite
morphism $T'\to T$ such that the base change $q':\kc':=\kc\times_TT'\to T'$ is a trivial family
of elliptic curves.

ii) For generic $y\in X$ there exist at least two smooth fibres $\kc_t,\kc_{t'}\subset \kc$ such that
$p(\kc_t)$ and $p(\kc_{t'})$ intersect transversally in $y$.

\smallskip

As in Section \ref{sec:BVexas}, we can always reduce to the situation that $q:\kc\to T$ comes with a zero-section,
the image of which  is a constant cycle curve in $X$ (we could even assume the image to be a rational curve).
The origin of a smooth fibre $\kc_t$ fixed by the zero section will be called $x_t$
and torsion in the fibres is considered with respect to this choice. 

\subsection{}
The technique to produce new constant cycle curves is by multiplying a given one
with respect to the additive structure of the fibres of $q:\kc\to T$. Under the above hypotheses one proves:

\begin{lem}\label{lem:transl}
Assume $D\subset \kc$ is a multi-section of $q:\kc\to T$ such that:

i) The image $C:=p(D)$ is a constant cycle curve of order $n$ and

ii) The intersection of $D$ with the generic fibres $\kc_t$ contains
a non-torsion point.

Then there exists an $N>0$ such that
the union of all constant cycle curves of order $\leq N$ is dense in the classical topology.
\end{lem}

\begin{proof} If $C\subset X$ is the image of a finite morphism $f:D\to X$
from an integral curve $D$, then the image
$\kappa_C\times k(\eta_D)$ of $\kappa_C$ under the base change map
$$\Ch^2(X\times k(\eta_C))\to\Ch^2(X\times k(\eta_D))$$
can be described as
the restriction of $[\Gamma_f-\{x_0\}\times D]\in\Ch^2(X\times D)$.

\smallskip

 Next consider fibrewise multiplication by $m$ for the family $q:\kc\to T$, which defines a morphism
$\mu_m:U\to U$ on some open set $U\subset\kc$. Denote the graphs of the two morphisms
$p:U\to X$ and  $p_m:=p\circ\mu_m:U\to X$
 by $\Gamma$ resp.\ $\Gamma_m$. We claim that,
possibly after shrinking $U$, there exists an $N_m>0$ such that 
$$N_m\cdot\left([\Gamma_m-\{x_0\}\times U]-m\cdot[\Gamma-\{x_0\}\times U]\right)=0\text{ in }\Ch^2(X\times U).$$
Indeed, the correspondence $[\Gamma_m]_*:\Ch^2(U)\to\Ch^2(X)$ maps
the class of a point $x\in\kc_t$ to $m\cdot[x]+(1-m)\cdot[x_t]$, where $x_t\in\kc_t$ is the origin of the elliptic curve
$\kc_t$. Since $[p(x_t)]=[x_0]$, one finds $[\Gamma_m-\{x_0\}\times U]_*=
m\cdot[\Gamma-\{x_0\}\times U]_*:\Ch^2(U)\to\Ch^2(X)$. Using Bloch--Srinivas \cite{BlochSrini},
one concludes the existence of $N_m$ (after shrinking $U$ if necessary).

\smallskip
 
Let now $D\subset \kc$ be as assumed. Fix $m$ and choose $N_m$ above such
that $n|N_m$. Then define $D(m^k):=\mu_{m^k}(D)=\mu_m^k(D)$ (or, more precisely,
$\overline{\mu_{m^k}(D\cap U)}$)
and $C(m^k):=p(D(m^k))=p_{m^k}(D)$. Since $\kappa_C\times k(\eta_D)$ is  the specialization
of $[\Gamma_p-\{x_0\}\times U]$ and similarly for all $C(m^k)$,
one finds
$$N_m\cdot\left(\kappa_{C(m^k)}\times k(\eta_D)-m\cdot \kappa_{C(m^{k-1})}\times k(\eta_D)\right)=0$$ in
$\Ch^2(X\times k(\eta_D))$. Since $n|N_m$ and, therefore, $N_m\cdot\kappa_C=0$,
one obtains by recursion that $N_m\cdot(\kappa_{C(m^k)}\times k(\eta_D))=0$ for all $k$.

The base change induced by $p_{m^k}:D\to C(m^k)$ factorizes
via
\begin{equation}\label{eqn:comp}
\Ch^2(X\times k(\eta_{C(m^k)}))\to\Ch^2(X\times k(\eta_{D(m^k)}))\to\Ch^2(X\times k(\eta_D)).
\end{equation}
Since $\deg(p:D(m^k)\to C(m^k))\leq \deg(p:\kc\to X)$ and $\deg(\mu_{m^k}:D\to D(m^k))\leq(D.\kc_t)$,
there exists an $N'>0$ independent of $m$ and $k$, such that the kernel of (\ref{eqn:comp}) is annihilated by $N'$.
This shows that the curves $C(m^k)$ are constant cycle curves of order $\leq N_m\cdot N'$. 

If the intersection of $D$ with the generic fibre $\kc_t$ contains a non-torsion point, then the
union $\bigcup D(m^k)\subset \kc$ of constant cycle curves of order $\leq N_m\cdot N'$
is dense in the Zariski topology and so is its image $\bigcup C(m^k)\subset X$.

To obtain density in the classical topology one needs to prove that the integers
$N_m$ can be chosen independent
of $m$.  Consider the map  
$\sigma:\kc\times_T\kc\to \kc\to X$ that is given as the composition
of the summation in the fibres $\kc_t$
with the projection  $p$. For points  $x,\,x'\in \kc$ in a smooth fibre $\kc_t$, we have
$[p(x+x')]=[p(x)]+[p(x')]-c_X$ in  $\Ch^2(X)$, as by assumption the class of the
image $p(x_t)$ of the origin  $x_t\in\kc_t$ is $c_X$.
Now consider the codimension two cycle 
$$\Gamma:=\Gamma_\sigma-\Gamma_{p\circ pr_1}-\Gamma_{p\circ pr_2}+
\{x_0\}\times\kc\times_T\kc$$
on $ X\times\kc\times_T\kc$. Then by
Bloch--Srinivas \cite{BlochSrini}, there exists an integer $N$ such that
the  $N\cdot[\Gamma]=0$ in $\Ch^2(X\times V)$ for some non-empty
open subset $V\subset \kc\times_T\kc$.  

Restriction to the  diagonal in $\kc\times_T\kc$
yields $$N\cdot\left([\Gamma_2] -2\cdot[\Gamma]+ [\{x_0\}\times W]\right)=0$$
in $\Ch^2(X\times W)$ for a certain non-empty open set
$W\subset \kc$ (the restriction of $V$ with the diagonal). By  induction on $m$ and by restricting the cycle
$\Gamma$ to the surface that is given as the image
of  ${\rm id}\times\mu_{m-1}:\kc\to\kc\times_T\kc$, $x\mapsto (x,(m-1)\cdot x)$,
one proves similarly
$$N\cdot\left([\Gamma_m]-m\cdot[\Gamma]+(m-1)\cdot[\{x_0\}\times W]\right)=0$$
in $\Ch^2(X\times W)$.

As before, this shows that all curves $C(m)$ are constant cycle curves of order
$\leq N\cdot N'$ (independent of $m$). If the generic fibre $\kc_t$ is viewed as a torus
$\RR^2/\ZZ^2$, then the intersection of $\bigcup D(m)$ with $\kc_t$ contains a
set of the form $\ZZ\cdot (x_1,x_2)\subset \RR^2/\ZZ^2$ with $x_1$ or $x_2$ non-torsion.
Therefore, the closure of $\kc_t\cap\bigcup D(m)$
is either the full torus $\kc_t=\RR^2/\ZZ^2$ or a finite union
 of translates of circles ${\mathbb S}^1\subset \kc_t=\RR^2/\ZZ^2$.

The  second possibility is ruled out, since the monodromy action
on $H_1(\kc_t,\ZZ)$ would then act via a finite group on the class of the
finite union of circles contradicting the fact that the monodromy
group is of finite index in ${\rm Sl}(2,\ZZ)$.
The argument can be made more explicit as follows. If the intersection is generically not dense,
then the components of its closure satisfy a linear equation of the form
$a\cdot x_1+b\cdot x_2=0$ with $a,b\in \ZZ$ not both zero. If assumed coprime, $a$ and $b$ are unique
up to sign. Such an equation exists for a countable union of real analytic subsets of $D$ over the
regular locus of $q$.  Thus, $(a,b)$ is locally constant up to sign and defines a section of the pull-back
of $R^1q_*\ZZ$ to a dense open subset of $D$ (or rather the appropriate double cover that accommodates
for the sign ambiguity). The latter contradicts the non-isotriviality of $\kc\times_TD\to D$.

The last part of the proof is inspired by a similar argument due to Chen and Lewis in \cite{CL}.
\end{proof}

\subsection{}
Let  $\kc_n\subset\kc$ be the closure of
the set of all $n$-torsion points $x\in\kc_t$ in smooth fibres $\kc_t$ and let $C_n:=p(\kc_n)$. Then
by Section \ref{sec:BVexas} the curve $C_n$ is a (possibly reducible) constant cycle curve of order $d|n$. Of course,
the curves $\kc_n$ cannot be used as input for Lemma \ref{lem:transl}, but the curve
$\widetilde \kc_n:= p^{-1}(p(\kc_n))$ might contain another component $D$ satisfying the assumptions i) (which is
automatic) and ii) of Lemma \ref{lem:transl}.
We will argue by contradiction and assume:

\smallskip

$(\ast)$ If $D\subset p^{-1}(p(\kc_n))=p^{-1}(C_n)$ is an irreducible component, then its intersection
with the generic fibres $\kc_t$ contains only torsion points.

\smallskip

Consider a multi-section $D\subset \kc$. We call $D$ \emph{flat} at a point  $x\in D$ if
locally analytically $D$ can be lifted to a {flat} section of $H^1_\CC:=R^1q_*\CC$.
E.g.\ the multi-sections $\kc_n$  are flat in all points of intersection with smooth fibres. Note that the notion makes sense for locally analytic sections as well.

\begin{lem}\label{lem:isotrivial}
Suppose $D\subset\kc$ is an irreducible multi-section of $q:\kc\to T$ and assume it
is flat locally around a point $x\in D$ which is a non-torsion point
in a smooth fibre $\kc_t$. Then $q:\kc\to T$ is isotrivial (contradicting the assumption i) in \ref{sec:notation}).
\end{lem}

\begin{proof}
The assumption immediately implies that $D$ is flat in all points of intersection with
smooth fibres. Consider the base change $q_D:\kc\times_TD\to D$. In addition to the pull-back of the
zero-section,
it comes with a natural flat section which is fibrewise non-torsion. 
This flat section of $q_D$ defines a section of $R^1q_{D*}\CC/R^1q_{D*}\ZZ$ (over the regular part)
 and, therefore, corresponds
to a monodromy invariant element of $\CC^2/\QQ^2$. But
for a non-isotrivial family the monodromy group is a finite index subgroup $\Gamma\subset{\rm Sl}(2,\ZZ)$
and thus $(\CC^2/\QQ^2)^\Gamma=0$. \end{proof}

Let now $x,x'\in\kc$ be generic points with $p(x)=p(x')$ and $x\in D\subset \kc$ a local multi-section
through $x$. Consider a component $D'\subset p^{-1}(p(D))$ through $x'$. A priori, $D$ might be
flat at $x$ without $D'$ being flat at $x'$. 
However, if $x$ is $n$-torsion in the fibre $\kc_t$, $t=q(x)$, then $x\in\kc_n$ and hence $x'\in\widetilde\kc_n$. So assuming $(\ast)$, the components of $\widetilde\kc_n$ containing $x'$ are again contained in some $\kc_m$
and are, therefore,  flat multi-sections of $q:\kc\to T$. Using a density argument this is enough to conclude the general case:

\begin{lem}\label{lem:flatalways}
Let $x,x'\in\kc$ be generic points with $p(x)=p(x')$. Assume $x\in D\subset\kc$ is a 
local analytic section through $x$ and $D'\subset p^{-1}(p(D))$ is a component through $x'$.
Under the assumption $(\ast)$ one has: If $D$ is flat at $x$, then $D'$ is flat at $x'$.
\end{lem}

\begin{proof}
Locally, flat sections are given by constant sections of $H^1_\CC|_\Delta=\CC^2\times\Delta$,
where $\Delta\subset T$ is a disk around a generic point of $T$.
For a torsion section $D$, which corresponds to a section contained in $\QQ^2\times\Delta$, assumption $(\ast)$
implies that $D'$ is again torsion and hence flat. In order to prove the assertion, it is enough to
prove that this holds true on an analytically dense set of flat local sections. However,
since $\kc\to T$ is not isotrivial, $R^0q_*(\Omega_{\kc/T})|_\Delta\subset H_\CC^1|_\Delta\otimes\ko$ does not admit
flat sections
%, i.e.\ the line $\CC\subset\CC^2$ given by the period of the fibre $\kc_t$ deforms with $t\in\Delta$,
and, therefore, $\Gamma(\Delta,R^1q_*\CC/R^1q_*\ZZ)$ can be identified with the set of flat
sections of $\kc|_\Delta\to\Delta$. Under this identification, the subset of sections contained
in $R^1q_*\QQ/R^1q_*\ZZ$ is analytically dense in   $\Gamma(\Delta,R^1q_*\CC/R^1q_*\ZZ)$.
\end{proof}

\noindent
{\it End of proof of Theorem \ref{thm:app}.} Choose generic points $x,x'\in\kc$ with $p(x)=p(x')$ and assume $x\in\kc_t$,
$t=q(x)$, is not torsion. Consider the two curves $E:=p(\kc_t)$ and $E':=p(\kc_{t'})$, $t'=q(x)$, in $X$.
Then we can assume that the two curves $E$ and $E'$ intersect
transversally in $y:=p(x)=p(x')\in E\cap E'$.
Now choose a flat local analytic section $x\in D\subset\kc$ of $q:\kc\to T$ such that 
$p(D)$ is tangent to $E'$ at $y$.   (This can always be achieved over points $t\in T$ with maximal
VHS. Indeed a local calculation shows that 
every non-vertical tangent direction $v\in T_x\kc$ in a point $x\in\kc$ over $t$
can be integrated to a flat local
section.) Then the component $D'\subset p^{-1}(p(D))$ through $x'$ is tangent to the fibre $\kc_{t'}$, $t'=q(x')$. Since by Lemma \ref{lem:flatalways}, $D'$ is also flat, in fact $D'\subset\kc_{t'}$.

But then $p^{-1}(E')$ contains $D$ and hence $D$ can be extended to an algebraic
multi-section of $\kc\to T$ through $x\in\kc$ which, moreover, is flat locally around $x$ and non-torsion.
Then, by Lemma \ref{lem:isotrivial}, $\kc\to T$ is isotrivial, which is excluded by assumption on $\kc\to T$.\qed

%%%%%%%%%%%%%%%%%%%%%%%%%%%

%\newpage

\end{document}